\documentclass[11pt,a4paper]{article}

\usepackage{titlesec}
\usepackage{fancyhdr}
\usepackage{a4wide}
\usepackage{graphicx}
\usepackage{float}
\usepackage{amssymb}
\usepackage{amsmath}
\usepackage{amsthm}
\usepackage{color}
\usepackage{mathrsfs}
\usepackage{array}
\usepackage{multido}
\usepackage{wrapfig}
\usepackage[T1]{fontenc}
\usepackage{inputenc}
\usepackage[english]{babel}
\usepackage{lmodern}
\usepackage{hyperref}
\usepackage{geometry}
\usepackage{changepage}
\changepage{0pt}{}{}{}{}{0pt}{}{0pt}{10pt}
\usepackage[numbers]{natbib}
\hypersetup{
pdfpagemode=UseThumbs,
pdftoolbar=true,        
pdfmenubar=true,        
pdffitwindow=false,     
pdfstartview={Fit},    
pdftitle={Spectrum of a diffusion operator with coefficient changing
sign over a small inclusion},    
pdfauthor={L. Chesnel, X. Claeys, S.A. Nazarov},     
pdfsubject={},  
pdfcreator={L. Chesnel, X. Claeys, S.A. Nazarov},   
pdfproducer={L. Chesnel, X. Claeys, S.A. Nazarov}, 
pdfkeywords={}, 
pdfnewwindow=true,      
colorlinks=true,       
linkcolor=magenta,          
citecolor=red,        
filecolor=cyan,      
urlcolor=blue           
}

\newcommand{\dsp}{\displaystyle}
\newcommand{\lbr}{\lbrack}
\newcommand{\rbr}{\rbrack}
\newcommand{\eps}{\varepsilon}
\newcommand{\om}{\omega}
\newcommand{\Om}{\Omega}
\newcommand{\mrm}[1]{\mathrm{#1}}
\newcommand{\mrB}{\mathrm{B}}

\newcommand{\bfx}{\boldsymbol{x}}
\newcommand{\bfn}{\boldsymbol{n}}
\newcommand{\bfxi}{\boldsymbol{\xi}}

\newcommand{\Cplx}{\mathbb{C}}
\newcommand{\N}{\mathbb{N}}
\newcommand{\R}{\mathbb{R}}

\renewcommand{\div}{\mrm{div}}
\newcommand{\mL}{\mrm{L}}
\newcommand{\mH}{\mrm{H}}
\newcommand{\mV}{\mrm{V}}

\newcommand{\mA}{\mrm{A}}
\newcommand{\mB}{\mrm{B}}

\renewcommand{\ker}{\mrm{ker}}
\newcommand{\coker}{\mrm{coker}}

\renewcommand{\dim}{\mrm{dim}}
\newcommand{\spec}{\mathfrak{S}}

\newtheorem{theorem}{Theorem}[section]
\newtheorem{lemma}{Lemma}[section]
\newtheorem{remark}{Remark}[section]

\newtheorem{corollary}{Corollary}[section]
\newtheorem{proposition}{Proposition}[section]

\newtheorem{Assumption}{Assumption}

\begin{document}

~\vspace{0.0cm}
\begin{center}
{\sc \bf\LARGE 
Spectrum of a diffusion operator with coefficient\\[4pt] changing
sign over a small inclusion}
\end{center}

\begin{center}
\textsc{Lucas Chesnel}$^1$, \textsc{Xavier Claeys}$^2$, \textsc{Sergei A. Nazarov}$^{3,\,4,\,5}$\\[16pt]
\begin{minipage}{0.91\textwidth}
{\small
$^1$ Centre de Math\'ematiques Appliqu\'ees, bureau 2029, \'Ecole Polytechnique, 91128 Palaiseau Cedex, France;\\
$^2$ Laboratory Jacques Louis Lions, University Pierre et Marie Curie, 4 place Jussieu, 75005 Paris, France; \\
$^3$ Faculty of Mathematics and Mechanics, St. Petersburg State University, Universitetsky prospekt, 28, 198504, Peterhof, St. Petersburg, Russia;\\
$^4$ Laboratory for mechanics of new nanomaterials, St. Petersburg State Polytechnical University, Polytekhnicheskaya ul, 29, 195251, St. Petersburg, Russia;\\
$^5$ Laboratory of mathematical methods in mechanics of materials, Institute of Problems of Mechanical Engineering, Bolshoj prospekt, 61, 199178, V.O., St. Petersburg, Russia;\\[10pt] 
E-mail: \texttt{chesnel@cmap.polytechnique.fr}, \texttt{claeys@ann.jussieu.fr}, \texttt{srgnazarov@yahoo.co.uk}\\[-14pt]
\begin{center}
(\today)
\end{center}
}
\end{minipage}
\end{center}
\vspace{0.4cm}

\noindent\textbf{Abstract.} 
We study a spectral problem $(\mathscr{P}^{\delta})$ for a diffusion like equation in a 3D domain $\Om$. The main originality lies in the presence 
of a parameter $\sigma^{\delta}$, whose sign changes on $\Om$, in the principal part of the operator we consider. More precisely, $\sigma^{\delta}$ 
is positive on $\Om$ except in a small inclusion of size $\delta>0$. Because of the sign-change of $\sigma^{\delta}$, for all $\delta>0$ the spectrum 
of $(\mathscr{P}^{\delta})$ consists of two sequences converging to $\pm\infty$. However, at the limit $\delta=0$, the small inclusion vanishes so 
that there should only remain positive spectrum for $(\mathscr{P}^{\delta})$. What happens to the negative spectrum? In this paper, we prove that 
the positive spectrum of $(\mathscr{P}^{\delta})$ tends to the spectrum of the problem without the small inclusion. On the other hand, we establish 
that each negative eigenvalue of $(\mathscr{P}^{\delta})$ behaves like $\delta^{-2}\mu$ for some constant $\mu<0$. We also show that the eigenfunctions 
associated with the negative eigenvalues are localized around the small inclusion. We end the article providing  2D numerical experiments 
illustrating these results.\\


\section{Introduction} 
\noindent \begin{wrapfigure}{r}{0.48\textwidth}
\vspace{-0.2cm}\centering\includegraphics{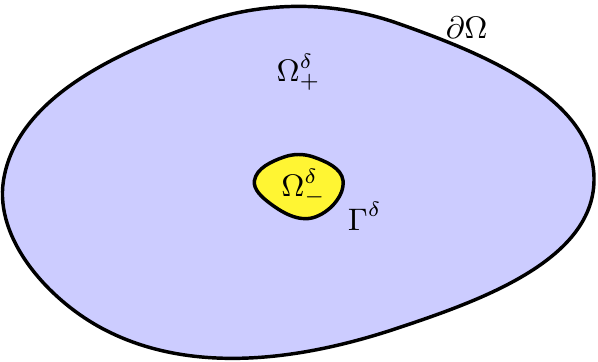}\vspace{-0.2cm}
\caption{Schematic view of the geometry.\label{GeneralGeom}}\vspace{-1.1cm}
\end{wrapfigure}
Let $\Om$ and $\Xi_-$ be three-dimensional domains, \textit{i.e.} bounded and connected open subsets of $\R^{3}$, with 
boundaries $\partial\Om$ and $\Gamma :=\partial\Xi_-$ that admit $\mathscr{C}^{\infty}$ regularity. Assume that $\Om$, $\Xi_-$ contain the 
origin $O$ and that there holds $\overline{\Xi_{-}}\subset\Om$. For $\delta\in(0;1]$, we introduce the sets (see Figure \ref{GeneralGeom})
$$
\begin{array}{lll}
\Om_{-}^{\delta} & := & \delta\,\Xi_{-}\\
\Om_{+}^{\delta} & := & \Om\setminus\overline{\Om_-^{\delta}}\\[3pt]
\Gamma^{\delta} & := & \partial \Om_{-}^{\delta}.
\end{array}
$$
Define $\sigma^{\delta}:\Omega\to\R$ by $\sigma = \sigma_{-}$ on $\overline{\Om\!\,^{\delta}_{-}}$ and 
$\sigma = \sigma_{+}$ on $\Om\!\,^{\delta}_{+}$, where $\sigma_{-}<0$ and 
$\sigma_{+}>0$ are constants. Throughout the paper, if $u$ is a measurable function on $\Om$, we shall denote 
$u_{\pm}:=u|_{\Om_{\pm}^{\delta}}$. For any open set $\om\subset \R^{3}$, the space  $\mL^{2}(\om)$ will 
refer to the set of square integrable functions defined on $\om$, equipped with the scalar product 
$(u,v)_{\om} =\int_{\om}u\overline{v}d\bfx$ and the norm $\Vert u\Vert_{\om} := \sqrt{(u,u)_{\om}}$. In the present article, we wish to study the following spectral problem  involving a Dirichlet boundary condition:
\begin{equation}\label{ExPb}
\begin{array}{|l}
\mbox{Find }(\lambda^{\delta},u^{\delta})\in\Cplx\times(\mH^{1}_{0}(\Om)\setminus\{0\})\mbox{ such that}\\[4pt]
-\div(\sigma^{\delta}\nabla u^{\delta}) = 
\lambda^{\delta}u^{\delta}\quad\mbox{ in }\Om.
\end{array}
\end{equation}
In Poblem (\ref{ExPb}), $\lambda^{\delta}$ is the spectral parameter. Moreover, $\mrm{H}^1_0(\Om)$ stands for the subspace of functions of the Sobolev space $\mrm{H}^1(\Om)$ 
vanishing on $\partial\Om$. It is endowed with the norm $\|u\|_{\mH^1_0(\Om)}:=\|\nabla u\|_{\Om}$. We regard $-\div(\sigma^{\delta}\nabla \cdot)$ as the unbounded operator $\mA^{\delta}:D(\mA^{\delta})\to \mL^{2}(\Omega)$ defined by 
\begin{equation}\label{ExOp}
\begin{array}{|l}
\mA^{\delta}\, v \;=\; -\mrm{div}(\sigma^{\delta}\nabla v)\\[6pt]
D(\mA^{\delta})\;:=\{v\in\mH^{1}_{0}(\Omega)\;\vert\; \mrm{div}(\sigma^{\delta}\nabla v)\in\mL^{2}(\Omega)\}.
\end{array}
\end{equation}
Since the interface $\Gamma^{\delta}=\partial \Om^{\delta}_{-}$ between the two subdomains is 
smooth, one can verify that when the contrast $\kappa_{\sigma} := \sigma_{-}/\sigma_{+}$ satisfies $\kappa_{\sigma}\neq -1$, $\mA^{\delta}$ fits the standard framework of \cite{LiMa68,RoSh63} for dealing with transmission problems. This leads to the following result (for a detailed discussion for this particular problem, see also \cite{Ramd99}).

\begin{proposition}\label{propoCompactResol}
Assume that $\kappa_{\sigma} = \sigma_{-}/\sigma_{+} \neq -1$. Then for all $\delta\in(0;1]$, the operator $\mA^{\delta}$ is densely defined, closed, self-adjoint 
and admits compact resolvent.
\end{proposition}

\noindent Therefore, for a fixed $\delta>0$, we can study the spectrum of $\mA^{\delta}$. Because $\sigma^{\delta}$ changes sign on $\Om$, this spectrum is not bounded from below nor from above. More precisely, we have the following results (see \cite{Ramd99} and \cite{BoRa02}).

\begin{proposition}\label{propoSptDesc}
Assume that $\kappa_{\sigma}\neq -1$. Then for all $\delta\in(0;1]$, the spectrum of $\mA^{\delta}$ consists in two sequences, one nonnegative and 
one negative, of real eigenvalues of finite multiplicity: 
\begin{equation}\label{spectrum delta}
\dots \lambda_{-n}^{\delta}\leq \dots\leq 
\lambda_{-1}^{\delta}< 0 \leq \lambda_{1}^{\delta}\leq \lambda_{2}^{\delta}\leq \dots 
\leq \lambda_{n}^{\delta} \dots\ .
\end{equation}
In the sequences above, the numbering is chosen so that each eigenvalue is repeated 
according to its multiplicity. Moreover, there holds $\dsp\lim_{n\to+\infty}\lambda_{\pm n}^{\delta}=\pm\infty$.
\end{proposition}

\noindent In Proposition \ref{propoCompactResol}, the assumption $\kappa_{\sigma}\neq-1$ is important and in the sequel, we should not depart from it. The case $\kappa_{\sigma}=-1$ is rather pathological and is beyond the scope of the present article.\\
\newline
Proposition \ref{propoSptDesc} indicates that for all  $\delta\in(0;1]$, $\spec(\mA^{\delta})$ (the spectrum of $\mA^{\delta}$) contains a sequence of eigenvalues which tends to $-\infty$. On the other hand, when $\delta$ goes to zero, the small inclusion vanishes so that the parameter $\sigma^{\delta}$ 
becomes strictly positive at the limit $\delta=0$. As a consequence, one could expect to obtain only positive spectrum for the operator $\mA^{\delta}$ when $\delta\to0$. The question we want to answer in this paper can be formulated as follows: what happens to the 
negative spectrum of $\mA^{\delta}$ when the small inclusion shrinks? \\
\newline
Problems of small inclusions or small holes have a long history in asymptotic analysis. The case where $\Om_{-}^{\delta}$  is a hole (with Dirichlet of Neumann condition 
on $\partial\Om_{-}^{\delta}$) in 2D/3D has been studied in detail in \cite{Na61,MaNP00} (see also the references therein). The configuration where $\Om_{-}^{\delta}$ contains a positive material with a concentrated mass has been investigated in \cite{SaPa84,OlSY91,Naza93,CaPV07} (see also the review \cite{LoPe03}). In this context,
the asymptotics of eigenpairs of elliptic operators has been considered in \cite{Na56,Na61,Na108}. 

\quad\\
A remarkable feature of Problem (\ref{ExPb}) is the change of sign of the parameter appearing in the principal part of the operator $\mA^{\delta}$. This is what makes it non-standard. In \cite{Na462,Na464,Na473}, the authors have examined the asymptotics of eigenpairs in situations where a sign changing coefficient arises in the compact part of the spectral problem under study. Yet, to our knowledge, asymptotics of the eigenpairs with a sign changing coming into play in the principal part of the operator has never been considered before.

\quad\\
The motivation for investigating Problem (\ref{ExPb}) comes from electromagnetics and in particular from 
the so-called \textit{surface plasmon polaritons}  as well as  \textit{metamaterials} that both offer a wide range of new technological perspectives. The surface plasmon polaritons are waves which 
propagate at the interface between a metal and a classical dielectric in the visible range. They appear because 
at optical frequencies, neglecting dissipation effects, the permittivity of a metal can be negative \cite{Ordal:83,BaDE03,CLMB05,ZaSM05,GrBo10}.  
The metamaterials (see \cite{Anan05} for an overview) are  artificial materials, made of small resonators periodically arranged so as to obtain macroscopic media with exotic permittivity $\eps$ and/or 
permeability $\mu$. In this field, one of the goals consists in achieving negative $\eps$ and/or $\mu$. We emphasize that in the present article, we consider a problem involving a material such that only one physical parameter ($\eps$ or $\mu$) takes negative values. When additionally there is a sign changing function in the right hand side of Equation (\ref{ExPb}), complex spectrum can appear and the analysis we propose does not apply.

\quad\\
The outline of the paper is the following. In Section \ref{SectionLimitProblems}, we start by presenting the features of the two 
limit operators $\mA^{0}$ and $\mB^{\infty}$ (see their definition in (\ref{FFOp}) and (\ref{NFOp})) which appear naturally in the 
study of the spectrum of $\mA^{\delta}$ when $\delta$ tends to zero. In Section \ref{SectionSourceTermpb}, we give an asymptotic 
expansion of the solution of the source term problem associated with (\ref{ExPb}) as $\delta\to0$.  In order to justify this 
asymptotic expansion, we establish uniform boundedness of the inverse of $\mA^{\delta}$ in terms of weighted norms involving the 
small parameter $\delta$. The proof of this important result is difficult because of the change of sign of the parameter $\sigma^{\delta}$. 
Our approach is based on the technique of overlapping cut-off functions introduced in \cite[Chap 2]{MaNP00}, \cite{Naza99}. Then, 
we make use of this uniform boundedness result in order to show that $(\mA^{\delta})^{-1}$ converges strongly to $(\mA^{0})^{-1}$ in 
the operator norm as $\delta\to 0$. This allows to prove directly, in Section \ref{StudyPositiveSpec}, that the positive part of 
the spectrum of $\mA^{\delta}$ converges to the spectrum of $\mA^{0}$. This will also imply that the negative eigenvalues of 
$\mA^{\delta}$ all diverge to $-\infty$. Section \ref{negativeSpectrum} then focuses on a sharper study of the negative part of 
the spectrum of $\mA^{\delta}$. In particular, we show that all negative eigenvalues admit a behaviour of the form $\delta^{-2}\mu$, 
with $\mu<0$. In the last section, we illustrate these theoretical results with 2D numerical experiments. The two main results 
of the paper, respectively for the positive and negative spectrum of $\mA^{\delta}$, are formulated in Theorem \ref{thmMajorPos} 
and Theorem \ref{NegSpectConv}.

\section{Limit problems}\label{SectionLimitProblems}
\noindent In the sequel, we will provide an asymptotic expansion of the eigenpairs of Problem (\ref{ExPb}) as $\delta$ tends to zero. This asymptotic expansion will involve the spectral parameters of some operators associated with two limit problems independent of $\delta$. The goal of the present section is to introduce these operators and to provide their main features. Before proceeding further, we introduce a set of cut-off functions which will be useful in our analysis. Let $\psi$ and $\chi$ be two elements of  $\mathscr{C}^{\infty}(\R,[0;1])$ such that
\begin{equation}\label{definition cut-off functions}
\psi(r)+\chi(r)=1,\qquad\psi(r)=1\ \mbox{ for }r\le1,\qquad \mbox{ and }\qquad\psi(r)=0\ \mbox{ for }r\ge2.
\end{equation}
For $t>0$, we shall denote $\psi_t$ and $\chi_t$ the functions (see Figure \ref{cut-off functions}) such that 
\begin{equation}\label{definition cut-off functions weigthed}
\psi_t(r)=\psi(r/t)\qquad\mbox{ and }\qquad\chi_t(r)=\chi(r/t).
\end{equation}

\begin{figure}[!ht]
\centering\includegraphics{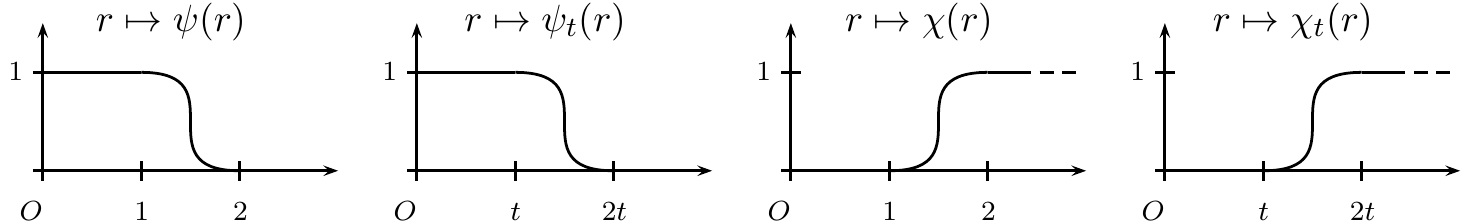}\vspace{-0.2cm}
\caption{Cut-off functions.\label{cut-off functions}}
\end{figure}
\noindent In order to simplify the presentation and without restriction, we shall assume that the domains $\Om$ and $\Xi_{-}$ are such that 
$\overline{\mrm{B}(O,2)}\subset\Om$, $\Xi_{-}\subset\overline{\mrm{B}(O,1)}$ so that the support of $\psi$ is included in $\Om$ and so that 
$\psi=1$ on $\Xi_{-}$. Here and throughout the paper, $\mrm{B}(O,d)$ denotes the open ball of $\R^3$, centered at $O$ and of radius $d>0$. 
If the domains $\Om$, $\Xi_{-}$ do not satisfy this assumption, we modify the cut-off functions accordingly.
\subsection{Far field operator}
\noindent As $\delta$ tends to zero, the small inclusion of negative material disappears. In other words, there holds $\sigma^{\delta}\to\sigma_{+}$ a.e. 
in $\Om$. This leads us to introduce the \textit{far field} operator $\mA^{0}:D(\mA^{0})\to \mL^{2}(\Omega)$ such that 
\begin{equation}\label{FFOp}
\begin{array}{|l}
\mA^{0} v \;=\; -\sigma_+\Delta v\\[6pt]
D(\mA^{0})\;:=\{v\in\mH^{1}_{0}(\Omega)\;\vert\; \Delta v\in\mL^{2}(\Omega)\}.
\end{array}
\end{equation}
Since $\partial\Om$ is smooth, $D(\mA^{0})$ coincides with $\mH^2(\Om)\cap\mH^1_0(\Om)$. 
Moreover, the operator $\mA^{0}$ is bijective from $D(\mA^{0})$ to $\mL^{2}(\Omega)$ and its spectrum forms a 
discrete sequence of eigenvalues:
\begin{equation}\label{eigenvalue bounded}
\begin{array}{l}
\mathfrak{S}(\mA^{0}) = \{\mu_{n}\}_{n\geq 1}\qquad\textrm{with}\qquad 0<\mu_1 <\mu_2\le\cdots\le\mu_n\ldots\underset{n\to+\infty}{\to}+\infty.
\end{array}
\end{equation}
In (\ref{eigenvalue bounded}), each eigenvalue is repeated according to its multiplicity. On the other hand, if $v_n$ is an eigenfunction associated with the eigenvalue $\mu_n$, then $v_n\in\mathscr{C}^{\infty}(\overline{\Om})$ (at least when $\partial\Om$ is of class $\mathscr{C}^{\infty}$).\\
\newline
In the sequel, we shall need sharp estimates for the behaviour of the eigenfunctions of $\mA^{0}$ at $O$. Following \cite{Kond67}, we shall express 
them in weighted norms. Let $\mathscr{C}^{\infty}_{0}(\overline{\Omega}\setminus\{O\})$ refer to the set of infinitely differentiable functions supported 
in $\overline{\Omega}\setminus\{O\}$. For $\beta\in \R$ and $k\geq 0$, we define the Kondratiev space $\mV^{k}_{\beta}(\Omega)$ as the completion of 
$\mathscr{C}^{\infty}_0(\overline{\Omega}\setminus\{O\})$ in the norm 
\begin{equation}\label{def norme poids ini}
\Vert v\Vert_{\mV^{k}_{\beta}(\Omega)} := \Big(\sum_{|\alpha|\le k}\int_{\Omega}
r^{2(\beta+\vert\alpha\vert-k)}\vert\partial^{\alpha}_{\bfx} v\vert^{2}\;d\bfx\;\Big)^{1/2}.
\end{equation}
In (\ref{def norme poids ini}), $r=|\bfx|$ denotes the distance to the origin $O$. We define the spaces $\mV^k_{\beta}(\Om_{\pm}^{\delta})$ 
like $\mV^k_{\beta}(\Om)$, replacing $\Om$ by $\Om_{\pm}^{\delta}$ in (\ref{def norme poids ini}). Of course, since $O\notin\Om^{\delta}_{+}$, 
the space $\mV^k_{\beta}(\Om_{+}^{\delta})$ coincides with $\mH^k(\Om_{+}^{\delta})$ for all $\beta\in\R$. To deal with homogeneous Dirichlet boundary condition, we shall consider functions belonging to the space $\mathring{\mV}^{1}_{\beta}(\Omega):=\{ v\in \mV^{1}_{\beta}(\Omega)
\;\vert\;v = 0\mbox{ on }\partial\Omega \}$. Then, we introduce the continuous operators 
$\mathcal{A}_{\beta}:D(\mathcal{A}_{\beta})\to \mV^0_{\beta}(\Om)$ such that
\begin{equation}\label{def op far field}
\begin{array}{|l}
\mathcal{A}_{\beta}v=-\sigma_+\Delta v\\[6pt]
D(\mathcal{A}_{\beta}):=\mV^2_{\beta}(\Om)\cap\mathring{\mV}^1_{\beta-1}(\Om).
\end{array}
\end{equation}
The proof of the following classical result can be found in the original paper \cite{Kond67} or, for example, in \cite[Chap 2]{NaPl94}, \cite[Chap 1]{MaNP00}.
\begin{proposition}\label{propoLaplaceBounded}
If $\beta\in(1/2;3/2)$ then $\mathcal{A}_{\beta}$ is an isomorphism. 
\end{proposition}

\subsection{Near field operator}\label{ParNearFieldOperator}
\noindent Introduce the \textit{rapid coordinate} $\bfxi:=\delta^{-1}\bfx$ and let $\delta$ tend to zero. Then, define the function $\sigma^{\infty}:\R^3\to\R$ 
such that $\sigma^{\infty}=\sigma_+$ in $\Xi_+:=\R^3\setminus\overline{\Xi}_{-}$ and $\sigma^{\infty}=\sigma_-$ in $\Xi_-$. In the sequel, the 
following  \textit{near field} operator $\mB^{\infty}:D(\mB^{\infty})\to \mL^{2}(\R^{3})$ will play a key role 
in the analysis:
\begin{equation}\label{NFOp}
\begin{array}{|l}
\mB^{\infty} w \;=\; -\div(\sigma^{\infty}\nabla w)\\[6pt]
D(\mB^{\infty})\;:=\{w\in\mH^1(\R^3)\;\vert\; \div(\sigma^{\infty}\nabla w)\in\mL^{2}(\R^3)\}.
\end{array}
\end{equation}
The description of $D(\mB^{\infty})$ is less classical than the one of $D(\mA^{0})$ because the sign of the parameter $\sigma^{\infty}$ 
is not constant on $\R^3$. Nevertheless, when $\kappa_{\sigma}\neq -1$, $\mB^{\infty}$ has elliptic regularity properties. In the next result, we consider the norm  $\Vert w\Vert_{\mH^{2}(\R^{3}\setminus\Gamma)}:=
\Vert w\Vert_{\mH^{2}(\Xi_{-})}+\Vert w\Vert_{\mH^{2}(\Xi_{+})}$, as well as the jump of normal derivative $\lbr \sigma\partial_{\bfn}w\rbr_{\Gamma}:=
\sigma_{+}\partial_{\bfn}w\vert_{\Gamma}^{+} - \sigma_{-}\partial_{\bfn}w\vert_{\Gamma}^{-}$, where $\bfn$ denotes the unit outward normal vector to $\Gamma$ directed from $\Xi_-$ to $\Xi_+$.

\begin{proposition}\label{PropDomainBinf}
Assume that $\kappa_{\sigma}\neq -1$. Then any $w\in D(\mB^{\infty})$ satisfies 
$w\in \mH^{2}(\Xi_{\pm})$ and $\lbr\sigma\partial_{\bfn}w \rbr_{\Gamma}=0$. Moreover, there exists a constant 
$C>0$ such that 
\begin{equation}\label{estimDomainBinf}
\|w\|_{\mH^2(\R^{3}\setminus\Gamma)}\le C\,(\|\div(\sigma^{\infty}\nabla w)\|_{\R^3}+\|w\|_{\R^3}),\qquad \forall w\in D(\mB^{\infty}).
\end{equation}
\end{proposition}
\begin{proof}
Clearly, any element $w\in \mH^{1}(\R^{3})$ that verifies $w\in \mH^{2}(\Xi_{\pm})$ and 
$\lbr\sigma\partial_{\bfn}w \rbr_{\Gamma}=0$ belongs to $D(\mB^{\infty})$. Now, let us pick some $w\in D(\mB^{\infty})$. 
The relation $\lbr\sigma\partial_{\bfn}w \rbr_{\Gamma}=0$ is classically obtained using Green's formula. Now, we wish to prove (\ref{estimDomainBinf}). Let $\psi$ be the cut-off function defined in (\ref{definition cut-off functions}). There holds $\div(\sigma^{\infty}\nabla(\psi w))=f$ with $f=\psi\,\div(\sigma^{\infty}\nabla  w)+2\sigma_+\nabla\psi
\cdot\nabla w+\sigma_+w\Delta\psi\in\mrm{L}^2(\R^{3})$ (notice that $\sigma^{\infty}=\sigma_+$ on the supports of $\nabla\psi$ and $\Delta\psi$). 
Since $\kappa_{\sigma}\neq -1$ and since $\Gamma = \partial\Xi_-$ is smooth, the results of \cite{LiMa68,RoSh63} allow to prove that $\psi w|_{\Xi_{\pm}}$ belongs to $\mH^2(\Xi_{\pm})$ with the estimate $\|\psi w\|_{\mH^2(\Xi_{\pm})}\le C\,(\|f\|_{\R^{3}}+\|w\|_{\Om})$. Here and in the sequel, $C>0$ denotes a constant independent of $w$ which can change from one line to another. Using interior regularity in $\Xi_{+}$, one establishes with classical techniques 
(see \cite{LiMa68}) that $\|f\|_{\R^{3}}\le C\,(\|\psi\,\div(\sigma^{\infty}\nabla w)\|_{\R^{3}}+\|w\|_{\Om})$. This implies
\begin{equation}\label{estim1domain}
\|\psi w\|_{\mH^2(\Xi_{\pm})}\le C\,(\|\psi \,\div(\sigma^{\infty}\nabla w)\|_{\R^{3}}+\|w\|_{\Om})
,\qquad\forall w\in D(\mrm{B}^{\infty}),
\end{equation} 
with a constant $C>0$ that does not depend on $w$. It remains to estimate $\div(\sigma^{\infty}\nabla(\chi w))=\sigma_+\Delta(\chi w)$. 
Thanks to Plancherel's theorem, we can write $\|\chi w\|_{\mH^2(\Xi_{+})}\le C\,(\|\div(\sigma^{\infty}\nabla(\chi w))\|_{\Xi_{+}}+\|\chi w\|_{\R^3})$. 
Using again interior regularity in $\Xi_{+}$, we deduce that $\|\chi w\|_{\mH^2(\Xi_{+})}\le C\,(\|\chi\div(\sigma^{\infty}\nabla w)\|_{\R^3}+\|w\|_{\R^3})$. 
This last inequality together with (\ref{estim1domain}) finally leads to (\ref{estimDomainBinf}) since $\psi+\chi=1$.
\end{proof}
\noindent Let us describe the spectrum $\mathfrak{S}(\mB^{\infty})$ of the operator $\mB^{\infty}$ when $\kappa_{\sigma}\neq -1$. 
According to the 3D version of \cite[Thm 5.2]{BoCC12}, for all $f\in\mrm{L}^2(\R^{3})$, we know that the problem 
``find $v\in\mH^1(\R^{3})$ such that $(\sigma^{\infty}\nabla v,\nabla v')_{\R^{3}}\pm i(v,v')_{\R^{3}}=(f,v')_{\R^{3}}$ 
for all $v'\in\mH^1(\R^{3})$'', has a unique solution. On the other hand, since $-\sigma_+\Delta\pm i\mrm{Id}:\mH^2(\R^3)\to\mrm{L}^2(\R^3)$ 
is bijective, we can prove that $\mB^{\infty}\pm i\mrm{Id}$ is bijective. Here, $\mrm{Id}$ denotes the identity of $\mrm{L}^2(\R^3)$. From 
\cite[Thm 4.1.7]{BiSo87}, we conclude that $\mB^{\infty}$ is self-adjoint. Moreover, observing that $\sigma^{\infty}=\sigma_+$ 
outside the compact region $\overline{\Xi}_{-}$, using again the 3D version of \cite[Thm 5.2]{BoCC12}, we can show 
that $\mB^{\infty}$ has the same continuous spectrum as the Laplace operator in $\R^3$. In other words, there holds 
$\mathfrak{S}_c(\mB^{\infty})=[0;+\infty)$. The interval $(-\infty;0)$ contains discrete spectrum only. 
Working as in \cite[Prop 4.1]{BoRa02}, we can build a sequence $(\zeta_n)_n$ of 
elements of $D(\mB^{\infty})$ such that $\lim_{n\to\infty}(\mB^{\infty}\zeta_n,\zeta_n)_{\R^3}=-\infty$ and $\|\zeta_n\|_{\R^3}=1$. 
According to \cite[Cor 4.1.5]{BiSo87}, this proves that the discrete spectrum of $\mB^{\infty}$ is equal to a sequence of eigenvalues 
\begin{equation}\label{eigenvalue unbounded}
\begin{array}{l}
\mathfrak{S}(\mrm{B}^{\infty})\setminus\overline{\R_{+}} = \{\mu_{-n}\}_{n\geq 1}\qquad\textrm{with} \qquad 
0>\mu_{-1}\ge\mu_{-2}\ge\cdots\ge\mu_{-n}\ldots\underset{n\to+\infty}{\to}-\infty,
\end{array}
\end{equation}
where each eigenvalue is repeated according to its multiplicity. In the following proposition, we establish that the 
eigenfunctions corresponding to the negative spectrum of $\mrm{B}^{\infty}$ are localized: they decay exponentially at infinity.

\begin{proposition}\label{proposition decompo eigenfunction Outter}
Assume that $\kappa_{\sigma}\neq -1$. For any eigenfunction $w$ of $\mrB^{\infty}$ 
associated with an eigenvalue $\mu\in \mathfrak{S}(\mrB^{\infty})\setminus\overline{\R_{+}}$ 
satisfying $\Vert w\Vert_{\R^{3}} = 1$, we have
\begin{equation}\label{estimEigenNF}
\int_{\R^{3}}(\,\vert w(\bfxi)\vert^{2}+\vert \nabla w(\bfxi)\vert^{2}\,)
\exp\left(\vert\bfxi\vert\sqrt{\vert\mu\vert/\sigma_{+}}\,\right)\,d\bfxi<+\infty.
\end{equation}
\end{proposition}

\begin{proof}
The proof of this proposition will rely on a technique used in \cite{KaNa00,CaDN10,Naza11}. Consider some $\gamma>0$ whose 
appropriate value will be fixed later on, and pick some $T\geq 2$. We introduce the weight function $\mathscr{W}_{\gamma}^T$ such that
\[
\mathscr{W}_{\gamma}^T(\bfxi) = 
\left\{ \begin{array}{lll} 
\exp(\gamma) & \mbox{ for } & \phantom{1\le\ }|\bfxi|\le 1\\ 
\exp(\gamma\vert\bfxi\vert)  & \mbox{ for } & 1\le|\bfxi|\le T \\
\exp(\gamma T) & \mbox{ for } & \phantom{1\le\,\, }|\bfxi|\ge T\,.
\end{array}\right.
\]
It is clear that $\mathscr{W}_{\gamma}^T$ is bounded and continuous. If $w$ is an eigenfunction of $\mB^{\infty}$ associated with the eigenvalue $\mu<0$, there holds $(\sigma^{\infty}\nabla w,\nabla w')_{\R^3} = \mu\,(w,w')_{\R^3}$ for all $w'\in\mH^1(\R^3)$. Choosing $w'=(\mathscr{W}_{\gamma}^T)^2w$ (observe that this function is indeed an element of $\mH^1(\R^3)$), we obtain 
\begin{equation}\label{estimExpoDecay1}
\begin{array}{lcl}
|\mu|\,\|\mathscr{W}_{\gamma}^Tw\|^2_{\R^3} & = & -(\sigma^{\infty}\nabla w,\nabla((\mathscr{W}_{\gamma}^T)^2w))_{\R^3}\\[4pt]
 & = & -(\sigma^{\infty}\mathscr{W}_{\gamma}^T\nabla w,\nabla(\mathscr{W}_{\gamma}^Tw))_{\R^3}-(\sigma^{\infty}\mathscr{W}_{\gamma}^T\nabla w,w\nabla \mathscr{W}_{\gamma}^T)_{\R^3}\\[4pt]
 & = & -(\sigma^{\infty}(\nabla(\mathscr{W}_{\gamma}^T w)-w\nabla \mathscr{W}_{\gamma}^T),\nabla(\mathscr{W}_{\gamma}^T w))_{\R^3}-(\sigma^{\infty}\mathscr{W}_{\gamma}^T\nabla w,w\nabla \mathscr{W}_{\gamma}^T)_{\R^3}\\[4pt]
 & = & -(\sigma^{\infty}\nabla(\mathscr{W}_{\gamma}^T w),\nabla(\mathscr{W}_{\gamma}^T w))_{\R^3}+(\sigma^{\infty}w\nabla \mathscr{W}_{\gamma}^T,w\nabla \mathscr{W}_{\gamma}^T)_{\R^3}.
\end{array}
\end{equation}
On $\Xi_{-}$, we have $\mathscr{W}_{\gamma}^T=\exp(\gamma)$. Therefore, (\ref{estimExpoDecay1}) rewrites as
\begin{equation}\label{estimExpoDecay2}
|\mu|\,\|\mathscr{W}_{\gamma}^Tw\|^2_{\R^3}+\sigma_{+}\,\|\nabla(\mathscr{W}_{\gamma}^Tw)\|^2_{\Xi_{+}} = \exp(2\gamma)|\sigma_{-}|\,\|\nabla w\|^2_{\Xi_{-}}+\sigma_{+}\,\|w\nabla \mathscr{W}_{\gamma}^T\|^2_{\R^{3}}.
\end{equation}
Add  $\sigma_{+}\|\nabla (\mathscr{W}_{\gamma}^T w)\|^2_{\Xi_{-}} = \sigma_{+}\exp(2\gamma)\|\nabla w\|^2_{\Xi_{-}}$ on each side of (\ref{estimExpoDecay2}), 
and use triangular inequality  $\Vert \nabla (\mathscr{W}_{\gamma}^T w)\Vert_{\R^{3}}^{2}\geq \frac{1}{2}
\Vert \mathscr{W}_{\gamma}^T\nabla w \Vert_{\R^{3}}^{2} - \Vert w \nabla \mathscr{W}_{\gamma}^T\Vert_{\R^{3}}^{2}$, to obtain 
\begin{equation}\label{estimExpoDecay3}
\vert \mu\vert \|w \mathscr{W}_{\gamma}^T\|^2_{\R^3}-2\sigma_{+}\|w\nabla \mathscr{W}_{\gamma}^T\|^2_{\R^{3}}+
\frac{\sigma_{+}}{2}\|\mathscr{W}_{\gamma}^T\nabla w\|^2_{\R^3} \leq \exp(2\gamma)\,(\sigma_{+}+\vert\sigma_{-}\vert)\,\|\nabla w\|^2_{\Xi_{-}}.
\end{equation}
We have $ \vert \nabla \mathscr{W}_{\gamma}^T\vert \leq \gamma \vert \mathscr{W}_{\gamma}^T\vert$, which implies
$\vert \mu\vert\|w \mathscr{W}_{\gamma}^T\|^2_{\R^3}-2\sigma_{+}\|w\nabla \mathscr{W}_{\gamma}^T\|^2_{\R^{3}}\geq (\vert \mu\vert-2\gamma^{2}\sigma_{+})
\,\|w \mathscr{W}_{\gamma}^T\|^2_{\R^3}$. From this, we conclude that
\[
\vert \mu\vert\|\mathscr{W}_{\gamma}^Tw\|_{\R^3}^{2}+\sigma_{+}\|\mathscr{W}_{\gamma}^T\nabla w\|_{\R^3}^{2} \le 
2\exp(2\gamma)\,(\sigma_{+}+\vert\sigma_{-}\vert)\,\|\nabla w\|^2_{\Xi_{-}},\qquad\textrm{for}\;\; 
\gamma = \frac{1}{2}\sqrt{\vert\mu\vert/\sigma_{+}}.
\]
There only remains to let $T$ tend to $+\infty$. Since the right-hand side of the inequality above 
is independent of $T$, this concludes the proof. \end{proof}
\noindent In the sequel, we will need to work with an operator analog to $\mrm{B}^{\infty}$ but considered in Sobolev spaces with weight at infinity. If $\bfxi\in\R^3$, we denote $\rho:=|\bfxi|$. For $\beta\in\R$, $k\geq 0$, we introduce the space $\mathcal{V}^{k}_{\beta}(\R^3)$ defined as the completion of the set $\mathscr{C}^{\infty}_0(\R^3)$ in the norm
\begin{equation}\label{def norme poids bis}
\Vert w\Vert_{\mathcal{V}^{k}_{\beta}( \R^3)} := \Big(\sum_{\vert\alpha\vert\le k}
\int_{ \R^3}(1+\rho)^{2(\beta+\vert\alpha\vert-k)}\vert\partial^{\alpha}_{\bfxi} w
\vert^{2} d\bfxi\;\Big)^{1/2}.
\end{equation}
We also define $\mathcal{V}^k_{\beta}(\Xi_{\pm}):=\{w|_{\Xi_{\pm}}\,\vert\,w\in\mathcal{V}^{k}_{\beta}(\R^3) \}$ (although
$\mathcal{V}^k_{\beta}(\Xi_{-})=\mH^k(\Xi_{-})$ for all $\beta\in\R$, as  $\Xi_{-}$ is bounded).  
Introduce the operator $\mathcal{B}_{\beta}:D(\mathcal{B}_{\beta})\to\mathcal{V}^0_{\beta}(\R^3)$ such that
\begin{equation}\label{def op inner field}
\begin{array}{|l}
\mathcal{B}_{\beta}w  =-\div(\sigma^{\infty}\nabla  w)\quad\forall w\in D(\mathcal{B}_{\beta})\\[6pt]
D(\mathcal{B}_{\beta}):=\{w\in\mathcal{V}^1_{\beta-1}(\R^3)\,|\,\div(\sigma^{\infty}\nabla  w)\in \mathcal{V}^0_{\beta}(\R^3)\}.
\end{array}
\end{equation}
Working as in the proof of Proposition \ref{PropDomainBinf} and using the Kondratiev theory, one can show that any element $w\in D(\mathcal{B}_{\beta})$ 
verifies $w\in \mV^{2}_{\beta}(\Xi_{\pm})$ and $\lbr \sigma\partial_{\bfn}w\rbr_{\Gamma} = 0$. Conversely it is straightforward to 
check that any $w\in \mV^{1}_{\beta-1}(\Xi_{\pm})$ satisfying these two conditions belongs to $D(\mathcal{B}_{\beta})$. Hence 
$D(\mathcal{B}_{\beta})$ is a closed subset of $\mathcal{V}^{2}_{\beta}(\R^{3}\setminus\Gamma)$, and $\mathcal{B}_{\beta}$ is continuous when equipping  
$D(\mathcal{B}_{\beta})$ with the norm  $\Vert w\Vert_{\mathcal{V}^{2}_{\beta}(\R^{3}\setminus\Gamma)} := \Vert w\Vert_{\mathcal{V}^{2}_{\beta}(\Xi_{+})}+ \Vert w\Vert_{\mathcal{V}^{2}_{\beta}(\Xi_{-})}$.
\begin{proposition}\label{proposition Fredholm Outter}
Assume that $\kappa_{\sigma}\ne-1$ and that $\beta\in(1/2;3/2)$. Then $\mathcal{B}_{\beta}$ is a Fredholm operator with 
$\mrm{ind}(\mathcal{B}_{\beta}):=\dim(\ker\,\mathcal{B}_{\beta})-\mrm{dim}(\coker\,\mathcal{B}_{\beta})=0$
\footnote{We recall that if $\mrm{X}$, $\mrm{Y}$ are two Banach spaces and if $L:\mrm{X}\to\mrm{Y}$ is a continuous linear map 
with closed range, then the cokernel of $L$ is defined as $\coker\,L:=(\mrm{Y}/\mrm{range}\,L)$.}.
\end{proposition}
\begin{proof}
Once the difficulty of the change of sign of $\sigma^{\infty}$ on the compact set $\Xi_{-}$ has been tackled thanks to \cite{CoSt85,BoCC12}, 
the proof of this proposition is rather classical and we will just sketch it. From the Kondratiev theory \cite{Kond67}, we know that for 
$\beta\in(1/2;3/2)$, the Laplace operator maps isomorphically  $\mathcal{V}^2_{\beta}(\R^3)$ onto $\mathcal{V}^0_{\beta}(\R^3)$. Using this 
result and inequality (\ref{estim1domain}), we obtain 
$$
\|w\|_{\mathcal{V}^{2}_{\beta}(\R^{3}\setminus\Gamma)}\le C\,(\|\div(\sigma^{\infty}\nabla w)\|_{\mathcal{V}^0_{\beta}(\R^3)}+\|w\|_{\Om}),
\qquad \forall w\in D(\mathcal{B}_{\beta}).
$$ 
Since the map $w\mapsto w|_{\Om}$ from $D(\mathcal{B}_{\beta})$ to $\mL^2(\Om)$ is compact, 
using the classical extension \cite{Tart87} of the well-known Peetre's lemma \cite{Peet61} (see also lemma 5.1 of \cite[Chap 2]{LiMa68}), we 
infer from the previous \textit{a priori} estimate that $\mathcal{B}_{\beta}$ has a kernel of finite dimension and that its range is closed 
in $\mathcal{V}^0_{\beta}(\R^3)$. As a consequence of the latter property, \cite[Thm 3.3.5]{BiSo87} ensures that $\coker\,\mathcal{B}_{\beta}$ 
is isomorphic to $\ker\,\mathcal{B}_{\beta}{}^{\ast}$ where $\mathcal{B}_{\beta}{}^{\ast}$ denotes the adjoint of $\mathcal{B}_{\beta}$. But the 
adjoint of $\mathcal{B}_{\beta}$ is the operator $\mathcal{B}_{2-\beta}$, which has a kernel of finite dimension because for $\beta\in(1/2;3/2)$, $2-\beta$ 
also belongs to $(1/2;3/2)$. Therefore $\coker\,\mathcal{B}_{\beta}$ is of finite dimension. Using the Kondratiev theory again, we can 
establish that for all $\beta\in(1/2;3/2)$, there holds $\ker\,\mathcal{B}_{\beta}= \cap_{\gamma\in(1/2;3/2)}\ker\,\mathcal{B}_{\gamma}$. This 
implies $\ker\,\mathcal{B}_{\beta}=\ker\,\mathcal{B}_{2-\beta}$, $\mrm{dim}(\coker\,\mathcal{B}_{\beta})=\mrm{dim}(\ker\,\mathcal{B}_{2-\beta})=
\mrm{dim}(\ker\,\mathcal{B}_{\beta})$ and so $\mrm{ind}(\mathcal{B}_{\beta})=0$. 
\end{proof}
\noindent  Depending on the parameter $\sigma_{\pm}$ and on the domain $\Xi_-$, it can happen that the operator $\mathcal{B}_{\beta}$ gets a non trivial 
(finite dimensional) kernel. We discard this possibility, considering an additional 
assumption
\begin{Assumption}\label{assumption1}
There exists $\beta\in (1/2;3/2)$ such that the operator $\mathcal{B}_{\beta}$ is injective.
\end{Assumption}

\noindent According to Proposition \ref{proposition Fredholm Outter} and the Kondratiev theory \cite{Kond67}, Assumption \ref{assumption1} 
implies that $\mathcal{B}_{\beta}:D(\mathcal{B}_{\beta})\to\mathcal{V}^0_{\beta}(\R^3)$ is an isomorphism for all $\beta\in(1/2,3/2)$ (and not just for one value of $\beta$). For a concrete case where this assumption is satisfied, one may for example consider the situation of Section \ref{Numerics}.

\begin{remark}\label{rmq kernel}
This assumption is interesting in its own and there exist results to check whether or not it holds for a given configuration. It is related to a question 
investigated by H. Poincar\'e in \cite{Poin97}. As it is done in the introduction of the very interesting paper \cite{KhPS07}, let us summarize Poincar\'e's 
problem with our notation. Let $w:\R^3\to\R$ be a continuous function whose restrictions to $\Xi_{\pm}$ are harmonic (like the elements of the kernel 
of $\mathcal{B}_{\beta}$). If we impose the total energy $\|\nabla w\|_{\R^3}$ to be equal to one, what is the minimum of $\|\nabla w\|_{\Xi_{-}}$? The answer is simple: this minimum is zero and it is attained for $w=w^0$ where $w^0=1$ in $\Xi_{-}$. Now, if we assume that $w$ satisfies both $\|\nabla w\|_{\R^3}=1$ 
and the orthogonality relation $(\nabla w,\nabla w^0)_{\R^3}=0$, what is the minimum of $\|\nabla w\|_{\Xi_{-}}$? Is it attained? It turns out that the minimum 
is indeed attained and is equal to some $m^1>0$. Moreover, if $w^1$ is a function which realizes this minimum, then there holds  $\partial_{\bfn}w^1|_{\Gamma}^+=
-m^1\partial_{\bfn}w^1|_{\Gamma}^- \mbox{ on }\partial\Xi_-$. In other words, $w^1$ belongs to the kernel of $\mathcal{B}_{\beta}$ when the contrast $\kappa_{\sigma}=\sigma_{-}/\sigma_{+}$ 
verifies $\kappa_{\sigma}=-m^1$. Continuing the process and imposing the orthogonality relations $(\nabla w,\nabla w^0)_{\R^3}=0$, $(\nabla w,\nabla w^1)_{\R^3}=0 \dots$, 
we can express all the values of the contrasts $\kappa_{\sigma}$ for which $\mathcal{B}_{\beta}$ fails to be injective in term of the extrema of the ratios of 
energies $\|\nabla w\|_{\Xi_{-}}/\|\nabla w\|_{\R^3}$. For more details concerning this question, we refer the reader to \cite{KhPS07,HePe12,PePu12,BoTr13,MR3195185}. 
\end{remark}
\begin{remark}\label{rmq plasmonics}
In the physical literature \cite{Maier}, values of $\kappa_{\sigma} = \sigma_{-}/\sigma_{+}$ leading to a non-trivial kernel for $\mathcal{B}_{\beta}$ are referred to as plasmonic eigenvalues.  Note that, although there is a growing interest on this 
so-called plasmonic eigenvalue problem, the latter is \textit{not the concern} of the present article. Here we assume that $\kappa_{\sigma}$ is fixed and not a plasmonic eigenvalue for the operator $\mathcal{B}_{\beta}$ (see Assumption \ref{assumption1} above), and we study the spectrum of the corresponding operator $\mA^{\delta}$.
\end{remark}
\begin{remark}
The study of the asymptotics of the eigenvalues of $\mA^{\delta}$ as $\delta$ goes to zero remains an open question when Assumption \ref{assumption1} does not hold.
\end{remark}

\noindent In this paper, we focus on the behaviour of eigenvalues and we leave aside the justification of 
asymptotics of eigenfunctions. We shall prove that if Assumption \ref{assumption1} holds, then for all $n\in\N^{\ast}:=\{1,2,\dots\}$, 
the eigenvalue $\lambda_{n}^{\delta}\ge0$ of $\mA^{\delta}$ converges to the corresponding eigenvalue $\mu_{n}>0$ of $\mA^{0}$. We 
shall also show that for all $n\in\N^{\ast}$, the eigenvalue $\lambda_{-n}^{\delta}<0$ of $\mA^{\delta}$ is such that $\delta^{2}\lambda_{-n}$ 
converges to $\mu_{-n}<0$, where $\mu_{-n}<0$ is the eigenvalue of $\mB^{\infty}$ defined in (\ref{eigenvalue unbounded}). Before doing 
that, in the next section, we investigate the source term problem associated with the original spectral Problem (\ref{ExPb}). In the 
process, we will establish a stability estimate which will reveal useful for the study of the initial problem.

\section{Asymptotic analysis for the source term problem}\label{SectionSourceTermpb}
\noindent Let $f$ be a source term in $\mrm{L}^2(\Om)$. In this section, we consider the problem  
\begin{equation}\label{PbAprioriEstimate}
\begin{array}{|l}
\mbox{Find }u^{\delta}\in\mH^1_{0}(\Om)\mbox{ such that }\\[4pt]
-\div(\sigma^{\delta}\nabla u^{\delta}) = f\quad\mbox{ in }\Om.
\end{array}
\end{equation}
For a fixed $\delta\in(0;1]$, using Proposition \ref{propoLaplaceBounded} and the 3D version of \cite[Thm 5.2]{BoCC12}, we can show, 
working as in the proof of Proposition \ref{proposition Fredholm Outter}, that Problem (\ref{PbAprioriEstimate}) is 
uniquely solvable if and only if it is injective. Under Assumption \ref{assumption1}, when the contrast satisfies 
$\kappa_{\sigma}\ne-1$, we will establish that it is injective (so that $(\mA^{\delta})^{-1}$ is well-defined) for $\delta$ small enough. Then, we provide, and justify with an error estimate, an asymptotic expansion of the solution $u^{\delta}$.

\subsection{A stability estimate}
\noindent We start by proving an important stability estimate for problem (\ref{PbAprioriEstimate}). Because the sign of $\sigma^{\delta}$ changes on $\Om$, this is a delicate procedure and the variational approach developed in \cite{BoCZ10,BoCC12} to establish Fredholm property for (\ref{PbAprioriEstimate}) seems useless here. Instead, we will employ a method introduced in \cite[Chap 2]{MaNP00}, \cite{Naza99} (see also \cite{ChCNSu} for an example 
of application in the context of negative materials), relying on the use of overlapping cut-off functions. To implement this technique, 
we need to introduce the Hilbert spaces $\mV^k_{\beta,\, \delta}(\om)$, $k\in\N$, $\beta\in\R$, over a domain $\om\subset \R^{3}$.
These spaces are defined as the completions of $\mathscr{C}^{\infty}(\overline{\om})$ for the weighted norms
\begin{equation}\label{weighted space bounded}
 \|v\|_{\mV^k_{\beta,\delta}(\om)}  := \Big(\sum_{|\alpha|\le k}\int_{\Omega}
(r+\delta)^{2(\beta+\vert\alpha\vert-k)}\vert\partial^{\alpha}_{\bfx} v\vert^{2}\;d\bfx\;\Big)^{1/2}.
\end{equation}
Observe that for any $\beta\in\R$ and $\delta>0$, the space $\mV^k_{\beta,\, \delta}(\Om)$ coincides with  $\mH^k(\Om)$ because the norm (\ref{weighted space bounded}) is equivalent to $\|\cdot\|_{\mH^k(\om)}$. However, the constants coming into play in this equivalence definitely depend on $\delta$ which is a crucial feature.
In the sequel, we shall consider the norm defined by 
$\|v\|_{\mV^2_{\beta,\delta}(\Om\setminus\Gamma^{\delta})}:= \|v\|_{\mV^2_{\beta,\delta}(\Om^{\delta}_{+})} + 
\|v\|_{\mV^2_{\beta,\delta}(\Om^{\delta}_{-})}$.

\begin{proposition}\label{proposition stability estimate}
Assume that $\kappa_{\sigma}\neq -1$ and that Assumption \ref{assumption1} holds. Then there is some $\delta_0>0$ such that the operator $\mA^{\delta}:D(\mA^{\delta})\to \mL^{2}(\Omega)$ defined by (\ref{ExOp}) is an isomorphism for all $\delta\in(0;\delta_0]$. Moreover, there holds $v\in \mH^{2}(\Omega_{\pm}^{\delta})$ for all $v\in D(\mA^{\delta})$ and,  for any $\beta\in(1/2;3/2)$, there exists $C_{\beta}>0$ independent of $\delta$ such that 
\begin{equation}\label{lemma stability estimate result}
C_{\beta}\leq \inf_{v\in D(\mA^{\delta})\setminus\{0\}}\frac{\| \mA^{\delta} v  \|_{\mV^0_{\beta,\,\delta}(\Om)}}{
\|v\|_{\mV^2_{\beta,\delta}(\Om\setminus\Gamma^{\delta})}},\qquad\forall \delta\in(0;\delta_0].
\end{equation}
\end{proposition}
\begin{proof}
We will construct explicitly the inverse of $\mA^{\delta}$. We first define a linear map $\mrm{R}^{\delta}:\mL^{2}(\Omega)\to D(\mrm{A}^{\delta})$ in the following manner.
Take any element $f\in\mL^{2}(\Omega)$. Decompose it in an inner and an outer contribution,
$$
f(\bfx)=g(\bfx)+G(\bfx/\delta)\qquad\mbox{ with }\quad g(\bfx)=\chi_{\sqrt{\delta}}(r)f(\bfx)\quad\mbox{ and }\quad G(\bfxi)=\psi_{\sqrt{\delta}}(\delta\rho)f(\delta\bfxi).
$$
Let us emphasize that according to the definition of $\chi$, $\psi$ (see (\ref{definition cut-off functions weigthed})), 
there holds $\chi_{\sqrt{\delta}}+\psi_{\sqrt{\delta}}=1$. Moreover, since $\psi_{\sqrt{\delta}}=1$ for $r\le\sqrt{\delta}$ and 
$\psi_{\sqrt{\delta}}=0$ for $r\ge 2\sqrt{\delta}$, $\mbox{supp}(G)$ (the support of $G$) is bounded in $\R^3$. As a consequence, for all $\eps\in\R$, $\delta>0$, we have $g\in\mV^0_{\beta-\eps,\delta}(\Om)$ and $G\in\mV^0_{\beta+\eps}(\R^3)$. Besides, for $\eps\ge 0$, the following estimates are valid:
\begin{equation}\label{estimation new source term}
\|g\|_{\mV^0_{\beta-\eps,\delta}(\Om)}=\|(\delta+r)^{\beta-\eps}g\|_{\Om\setminus\overline{\mrm{B}(O,\sqrt{\delta})}} \le 
C\delta^{-\eps/2}\,\|(\delta+r)^{\beta}g\|_{\Om\setminus\overline{\mrm{B}(O,\sqrt{\delta})}} \le \delta^{-\eps/2}\,\|f\|_{\mV^0_{\beta,\,\delta}(\Om)}
\end{equation}
\begin{equation}\label{estimation new source termNF}
\begin{array}{llcl}
\mbox{ and } & \|G\|_{\mathcal{V}^0_{\beta+\eps}(\R^3)} & = &\|(1+\rho)^{\beta+\eps}G\|_{\mrm{B}(O,2/\sqrt{\delta})}  \\[4pt]
 & & \le & C\,\delta^{-\eps/2}\|(1+\rho)^{\beta}G\|_{\mrm{B}(O,2/\sqrt{\delta})}\  \le\   C\,\delta^{-\beta-3/2-\eps/2}\,\|f\|_{\mV^0_{\beta,\,\delta}(\Om)}.
\end{array}
\end{equation}
Here and in the sequel, $C>0$ denotes a constant, which can change from one line to another, depending on $\beta$, $\eps$ 
but not on $\delta$. The last inequality in (\ref{estimation new source termNF}) has been obtained making the change of 
variables $\bfx=\delta\bfxi$. This explains the appearance of the term $\delta^{-3/2-\beta}$. Let us choose $\eps>0$ small 
enough so that $[\beta-\eps;\beta+\eps]\in(1/2;3/2)$. Let us consider $v\in D(\mathcal{A}_{\beta-\eps})$, $V\in D(\mathcal{B}_{\beta+\eps})$ 
such that
\[
\mathcal{A}_{\beta-\eps}v=g,\qquad\qquad\qquad \mathcal{B}_{\beta+\eps}V=\delta^{2}G,
\]
where $\mathcal{A}_{\beta-\eps}$ and $\mathcal{B}_{\beta+\eps}$ were introduced respectively in (\ref{def op far field}) and (\ref{def op inner field}). 
Proposition \ref{propoLaplaceBounded} and Proposition \ref{proposition Fredholm Outter} coupled with Assumption \ref{assumption1} 
ensure that $v$, $V$ are well-defined and that they satisfy the estimates
\begin{equation}\label{estimate limit geometries}
\|v\|_{\mV^2_{\beta-\eps}(\Om)} \le C\,\|g\|_{\mV^0_{\beta-\eps}(\Om)}\le C\,\|g\|_{\mV^0_{\beta-\eps,\delta}(\Om)},\qquad\quad \|V\|_{\mathcal{V}^2_{\beta+\eps}(\R^{3}\setminus\Gamma)} \le C\,\delta^{2}\|G\|_{\mathcal{V}^0_{\beta+\eps}(\R^3)}.
\end{equation}
Finally, we define the operator $\mrm{R}^{\delta}$ such that for all $f\in\mL^2(\Om)$, $\mrm{R}^{\delta}f=\hat{u}^{\delta}$ with
\begin{equation}\label{constructionUhat}
\hat{u}^{\delta}(\bfx):=\chi_{\delta}(r)v(\bfx)+\psi(r)V_{\delta}(\bfx),\qquad\mbox{ where }\qquad V_{\delta}(\bfx):=V(\bfx/\delta).
\end{equation}
Since $V\in D(\mathcal{B}_{\beta})$, there holds $\mrm{R}^{\delta}f\in \mH^{2}(\Omega_{\pm}^{\delta})$. Moreover, we have $\lbr\sigma\partial_{\bfn}V\rbr_{\Gamma}=0$ so that $\mrm{R}^{\delta}f=\hat{u}^{\delta}$ belongs to $D(\mrm{A}^{\delta})$. Now observe that, 
since $r\leq r+\delta\leq 2r$ on $\mrm{supp}(\chi_{\delta})$, there exists a constant $C>0$
independent of $\delta$ such that $\Vert \chi_{\delta} v\Vert_{\mV^{2}_{\beta,\delta}(\Omega)}\leq 
C\Vert \chi_{\delta} v\Vert_{\mV^{2}_{\beta}(\Omega)}$. Hence, using (\ref{estimation new source term})-(\ref{estimate limit geometries}) 
with $\eps=0$ we can check (work as in \cite[Lem 2]{Naza93}) that
\begin{equation}\label{estimateUhat}
\Vert \mrm{R}^{\delta}f\Vert_{\mV^{2}_{\beta,\delta}(\Omega\setminus\Gamma^{\delta})}= \|\hat{u}^{\delta}\|_{\mV^2_{\beta,\delta}(\Om_{-}^{\delta})}+
\|\hat{u}^{\delta}\|_{\mV^2_{\beta,\delta}(\Om_{+}^{\delta})} \le C\,\|f\|_{\mV^0_{\beta,\,\delta}(\Om)}.
\end{equation}
We wish to prove that $\mrm{R}^{\delta}$ is an approximate inverse of $\mA^{\delta}$. Practically, we are going to show that $\hat{u}^{\delta}$  satisfies $\div(\sigma^{\delta}\nabla \hat{u}^{\delta}) = f^{\delta}$ for some 
source term $f^{\delta}$ close to $f$. A direct computation yields 
\begin{equation}\label{commutator hat}
\begin{array}{lcl}
-\div(\sigma^{\delta}\nabla \hat{u}^{\delta}) \hspace{-0.2cm}& \hspace{-0.2cm}= &\hspace{-0.2cm}  -\big[\div(\sigma^{\delta}\nabla\cdot),\chi_{\delta}\big]v-
\big[\div(\sigma^{\delta}\nabla\cdot),\psi\big]V_{\delta}-\chi_{\delta}\,\div(\sigma^{\delta}\nabla v^{\delta})-\psi\,\div(\sigma^{\delta}\nabla V_{\delta})\\[0.6em]
 \hspace{-0.2cm}& \hspace{-0.2cm}= &\hspace{-0.2cm} -\big[\div(\sigma^{\delta}\nabla\cdot),\chi_{\delta}\big]v-\big[\div(\sigma^{\delta}\nabla\cdot),\psi\big]V_{\delta}+
\chi_{\delta}\,\chi_{\sqrt{\delta}}\,f+\psi\,\psi_{\sqrt{\delta}}\,f\\[0.6em]
\hspace{-0.2cm} &\hspace{-0.2cm} = &\hspace{-0.2cm} -\big[\div(\sigma^{\delta}\nabla \cdot),\chi_{\delta}\big]v-\big[\div(\sigma^{\delta}\nabla\cdot),
\psi\big]V_{\delta}+f.
\end{array}
\end{equation}
In the above equalities, the commutator $[A,B]$ is defined by $[A,B]=AB-BA$. Since we have $\delta\le r \le 2\delta$ on the support of $\big[\div(\sigma^{\delta}\nabla\cdot),\chi_{\delta}\big]$ and $1\le r \le 2$ on the support of $\big[\div(\sigma^{\delta}\nabla\cdot),\psi\big]$, we find that $\div(\sigma^{\delta}\nabla \hat{u}^{\delta})\in\mV^0_{\beta,\,\delta}(\Om)$. Now, introduce the operator $\mrm{K}^{\delta}:\mV^0_{\beta,\,\delta}(\Om)\to\mV^0_{\beta,\,\delta}(\Om)$ such that for all $f\in\mV^0_{\beta,\,\delta}(\Om)$, $\mrm{K}^{\delta}f$ denotes the function verifying
\[
-\div(\sigma^{\delta}\nabla \hat{u}^{\delta}) = f + \mrm{K}^{\delta}f.
\]
From (\ref{commutator hat}), we know that $\mrm{K}^{\delta}f = -\big[\div(\sigma^{\delta}\nabla\cdot),\chi_{\delta}\big]v-\big[\div(\sigma^{\delta}\nabla\cdot),\psi\big]V_{\delta}$. Let us evaluate the norm of $\mrm{K}^{\delta}$. First, remark that 
\begin{equation}\label{PropertChiDelta}
|\nabla\chi_{\delta}|\le C\,\delta^{-1}\qquad\mbox{ and }\qquad|\Delta\chi_{\delta}|\le C\,\delta^{-2}.
\end{equation}
Defining, for $t>0$, the set $\mathbb{Q}^{t}:=\{\bfx\in\R^3\,|\,t<|\bfx|<2t\}$, we can write 
\begin{equation}\label{estimateReminder1}
\begin{array}{lcl}
\|\big[\div(\sigma^{\delta}\nabla\cdot),\chi_{\delta}\big]v\|_{\mV^0_{\beta,\,\delta}(\Om)} & \le &
\|(r+\delta)^{\beta}\sigma_+\nabla\chi_{\delta}\cdot\nabla v\|_{\Om}+\|(r+\delta)^{\beta}\sigma_+\Delta\chi_{\delta}\,v\|_{\Om}\\[4pt]
& \le & C\,(\delta^{-1}\|(r+\delta)^{\beta}\nabla v\|_{\mathbb{Q}^{\delta}}+\delta^{-2}\|(r+\delta)^{\beta} v\|_{\mathbb{Q}^{\delta}})\\[4pt]
& \le & C\,\delta^{\eps}\,(\|r^{\beta-\eps-1}\nabla v\|_{\mathbb{Q}^{\delta}}+\|r^{\beta-\eps-2} v\|_{\mathbb{Q}^{\delta}})\\[4pt]
& \le & C\,\delta^{\eps}\,\|v\|_{\mV^2_{\beta-\eps}(\Om)} \ \le\ C\,\delta^{\eps/2}\,\|f\|_{\mV^0_{\beta,\,\delta}}.
\end{array}
\end{equation}
In (\ref{estimateReminder1}), the last inequality comes from (\ref{estimate limit geometries}), (\ref{estimation new source term}). Proceeding similarly, we find 
$\|\big[\div(\sigma^{\delta}\nabla\cdot),\psi\big]V_{\delta}\|_{\mV^0_{\beta,\,\delta}(\Om)} \le C\,\delta^{\eps/2}\,\|f\|_{\mV^0_{\beta,\,\delta}}$. Therefore, 
for all $f\in\mV^0_{\beta,\,\delta}(\Om)$, there holds $\|\mrm{K}^{\delta}f\|_{\mV^0_{\beta,\,\delta}(\Om)}\le C\,\delta^{\eps/2}\,\|f\|_{\mV^0_{\beta,\,\delta}(\Om)}$. 
This proves that the norm of $\mrm{K}^{\delta}:\mV^0_{\beta,\,\delta}(\Om)\to\mV^0_{\beta,\,\delta}(\Om)$ tends to zero when $\delta\to0$. As a consequence, 
$\mrm{Id}+\mrm{K}^{\delta}$ is invertible for $\delta$ small enough.\\
\newline 
From this discussion, since $\mrm{A}^{\delta}\cdot \mrm{R}^{\delta}f = (\mrm{Id}+\mrm{K}^{\delta})f$, we conclude that the operator $\mA^{\delta}$ admits $\mrm{R}^{\delta}\cdot(\mrm{Id}+\mrm{K}^{\delta})^{-1}$ as 
a continuous inverse for $\delta$ small enough. This proves that there exists $\delta_0>0$ such that $\mA^{\delta}$ is an isomorphism for all $\delta\in(0;\delta_0]$. Formula (\ref{constructionUhat}) implies that if $v\in D(\mA^{\delta})$ verifies $\mA^{\delta}v=f$ and if $\mA^{\delta}$ is invertible, then we have $v= \mrm{R}^{\delta}\cdot(\mrm{Id}
+\mrm{K}^{\delta})^{-1}f\in \mH^{2}(\Omega_{\pm}^{\delta})$. Finally, we obtain (\ref{lemma stability estimate result}) from estimates (\ref{estimateUhat}) and (\ref{estimateReminder1}) which provide respectively uniform bounds of $\mrm{R}^{\delta}$ and $(\mrm{Id}+\mrm{K}^{\delta})^{-1}$.
\end{proof}

\subsection{Asymptotic expansion for the source term problem}\label{sectionSourceTerm}
\noindent Now that we know that problem (\ref{PbAprioriEstimate}) is uniquely solvable for $\delta$ small enough, we can think of providing an asymptotic expansion of its solution $u^{\delta}$ as $\delta$ goes to zero. This is precisely the goal of the present section. When $\delta\to0$, as previously observed, the inclusion of negative material disappears. This leads us naturally to consider the following problem
\begin{equation}\label{source term limit}
\begin{array}{|l}
\mbox{Find }v\in\mH^1_{0}(\Om)\mbox{ such that }\\[4pt]
-\sigma_+\Delta v = f\quad\mbox{ in }\Om.
\end{array}
\end{equation}
Here, $f\in\mrm{L}^2(\Om)$ is the same source term as the one of (\ref{PbAprioriEstimate}). We have $\mrm{L}^2(\Om) \subset \mV^0_{\beta}(\Om)$ for all $\beta\in(1/2;3/2)$. Therefore, Proposition 
\ref{propoLaplaceBounded} ensures that Problem (\ref{source term limit}) has a unique solution $v = (\mA^{0})^{-1}f$ satisfying, for all $\beta\in(1/2;3/2)$, $v\in \mV^{2}_{\beta}(\Omega)$. In addition, since $f\in\mrm{L}^2(\Om)$, it is known 
(see \cite[Chap 1]{MaNP00}) that $v-v(0)$ belongs to $\mV^2_{0}(\Om)$ with the estimate
\begin{equation}\label{estimate far field source term}
|v(0)|+\|v-v(0)\|_{\mV^2_{0}(\Om)} \le C\,\|f\|_{\Om}.
\end{equation}
In next proposition, we prove that for $\delta$ small enough, $v$ is a good approximation of $u^{\delta}$, the unique solution to Problem (\ref{PbAprioriEstimate}). To proceed, from $v$, we define 
another function $\mathcal{R}^{\delta}f$ by the formula 
\begin{equation}\label{definition approximation source term}
\mathcal{R}^{\delta}f := \chi_{\delta}\,v+\psi_{\delta}\,v(0),
\end{equation}
where $\chi_{\delta}$, $\psi_{\delta}$ were introduced in (\ref{definition cut-off functions weigthed}). The operator
$\mathcal{R}^{\delta}$ is clearly linear, and estimate (\ref{estimate far field source term}) indicates that it maps
continuously $\mL^{2}(\Omega)$ into $\mH^{2}(\Omega\setminus\Gamma^{\delta})$. Observe in addition 
that, since $\mathcal{R}^{\delta}v$ is constant in a neighbourhood of $\Omega_{-}^{\delta}$, we actually have $\mathcal{R}^{\delta}f\in D(\mA^{\delta})$. The following result is obtained by observing that $\mathcal{R}^{\delta}$ approximates both $(\mA^{\delta})^{-1}$ and $(\mA^{0})^{-1}$.
\begin{proposition}\label{propoEstimResolvantes}
Assume that $\kappa_{\sigma}\neq -1$ and that Assumption \ref{assumption1} holds. Then for any 
$\beta\in(1/2,3/2)$, there exist constants $C_{\beta}$, $\delta_0>0$ independent of $\delta$ such  that 
\begin{equation}\label{error estimate result}
\sup_{f\in\mL^{2}(\Omega)\setminus\{0\}}\frac{\Vert (\mA^{0})^{-1} f- (\mA^{\delta})^{-1}f\Vert_{\mV^{2}_{\beta,\delta}(\Omega\setminus\Gamma^{\delta})  }}{
\Vert f\Vert_{\Om}}\leq C_{\beta}\,\delta^{\beta}, \qquad\forall \delta\in(0;\delta_0].
\end{equation}
\end{proposition}
\begin{proof}
Take any $f\in\mL^{2}(\Omega)$ and for $\delta$ small enough, set $u^{\delta} := (\mA^{\delta})^{-1}f$, $\hat{u}^{\delta}:= \mathcal{R}^{\delta}f$, $\tilde{v}:=v-v(0)$ where $v$ is the solution to (\ref{source term limit}). Since $1=\chi_{\delta}+\psi_{\delta}$, we have $\nabla\psi_{\delta}=-\nabla\chi_{\delta}$ and $\Delta\psi_{\delta}=-\Delta\chi_{\delta}$. This implies
\[
\begin{array}{lcl}
-\div(\sigma^{\delta}\nabla(u^{\delta}-\hat{u}^{\delta})) & = & f+\div(\sigma^{\delta}\nabla \hat{u}^{\delta})\\[0.4em]
 & = & f+\sigma_+\chi_{\delta}\Delta v+2\sigma_+\nabla\chi_{\delta}\cdot\nabla \tilde{v}+\sigma_+\tilde{v}\Delta\chi_{\delta}\\[0.4em]
 & = & f(1-\chi_{\delta})+2\sigma_+\nabla\chi_{\delta}\cdot\nabla \tilde{v}+\sigma_+\tilde{v}\Delta\chi_{\delta}:=\hat{f}^{\delta}.
\end{array}
\]
What precedes can be rewritten $(\mrm{Id} - \mA^{\delta}\cdot\mathcal{R}^{\delta})f = \hat{f}^{\delta}$. Using (\ref{PropertChiDelta}), we find
\[
\begin{array}{lcl}
  \|\hat{f}^{\delta}\|_{\mV^0_{\beta,\,\delta}(\Om)} & \le & \|f(1-\chi_{\delta})\|_{\mV^0_{\beta,\,\delta}(\Om)}+
\|2\sigma_+\nabla\chi_{\delta}\cdot\nabla\tilde{v}\|_{\mV^0_{\beta,\,\delta}(\Om)}
+\|\sigma_+\tilde{v}\Delta\chi_{\delta}\|_{\mV^0_{\beta,\,\delta}(\Om)}\\[0.4em]
 & \le & C\,(\|(r+\delta)^{\beta}f\|_{\mrm{B}(O,2\delta)}+\delta^{-1}\|(r+\delta)^{\beta}\nabla \tilde{v}\|_{\mathbb{Q}^{\delta}}
+\delta^{-2}\|(r+\delta)^{\beta}\tilde{v}\|_{\mathbb{Q}^{\delta}})\\[0.4em]
  & \le  & C\,(\delta^{\beta}\,\|f\|_{\Om}+\delta^{\beta}\,\|\tilde{v}\|_{\mV^2_{0}(\Om)})\ \le\ C\,\delta^{\beta}\,\|f\|_{\Om}.
\end{array}
\]
To obtain the last line of the above inequalities, we have used (\ref{estimate far field source term}). 
This proves that $\Vert (\mrm{Id} - \mA^{\delta}\cdot\mathcal{R}^{\delta})f\Vert_{\mV^{0}_{\beta,\delta}(\Omega)}
\leq C\delta^{\beta}\Vert f\Vert_{\Omega}$ for all $f\in\mL^{2}(\Omega)$. Now, observe that 
$\mA^{\delta}\cdot\mathcal{R}^{\delta}f = \mA^{0}\cdot\mathcal{R}^{\delta}f$. Therefore, we also have
$\Vert (\mrm{Id} - \mA^{0}\cdot\mathcal{R}^{\delta})f\Vert_{\mV^{0}_{\beta,\delta}(\Omega)}
\leq C\delta^{\beta}\Vert f\Vert_{\Omega}$ for all $f\in\mL^{2}(\Omega)$. From Proposition \ref{proposition stability estimate} and Lemma \ref{Stability2} hereafter, we deduce 
\begin{equation}\label{estimTriang}
\begin{array}{ll}
 & \Vert (\mA^{0})^{-1}f - (\mA^{\delta})^{-1}f\Vert_{\mV^{2}_{\beta,\delta}(\Omega\setminus\Gamma^{\delta})}\\[6pt]
 \leq & \Vert (\mA^{0})^{-1}f -\mathcal{R}^{\delta}f\Vert_{\mV^{2}_{\beta,\delta}(\Omega\setminus\Gamma^{\delta})} +
\Vert(\mA^{\delta})^{-1}f -\mathcal{R}^{\delta}f\Vert_{\mV^{2}_{\beta,\delta}(\Omega\setminus\Gamma^{\delta})}\\[6pt]
 \leq & C\Vert f -\mA^{0}\cdot\mathcal{R}^{\delta}f\Vert_{\mV^{0}_{\beta,\delta}(\Omega)} +
C\Vert f -\mA^{\delta}\cdot\mathcal{R}^{\delta}f\Vert_{\mV^{0}_{\beta,\delta}(\Omega)}\leq C\delta^{\beta}\Vert f\Vert_{\Omega}.
\end{array}
\end{equation}
In (\ref{estimTriang}), $C>0$ is a constant, which can change from one line to another, independent of $\delta, f$. With (\ref{estimTriang}), we finally obtain (\ref{error estimate result}).
\end{proof}
\noindent The following lemma is a technical result needed in the proof of Proposition \ref{propoEstimResolvantes}. It claims that the stability property (\ref{lemma stability estimate result}) for $(\mA^{\delta})^{-1}$ is also satisfied by $(\mA^{0})^{-1}$. 
\begin{lemma}\label{Stability2}
For any $\beta\in(1/2,3/2)$, there exist constants $C_{\beta},\delta_{0}>0$ independent of $\delta$ such that 
$$
C_{\beta}\leq \inf_{v\in D(\mA^{0})\setminus\{0\}}\frac{\Vert \mA^{0}v\Vert_{\mV^{0}_{\beta,\delta}(\Omega)}}{
\Vert v\Vert_{\mV^{2}_{\beta,\delta}(\Omega\setminus\Gamma^{\delta})}},\qquad\forall \delta\in(0;\delta_0].
$$
\end{lemma}

\begin{proof}
Decompose any element 
$v\in D(\mA^{0})$ as $v(\bfx) = \chi_{\delta}(\bfx)v(\bfx) + V_{\delta}(\bfx/\delta)$ where
$V_{\delta}(\bfx/\delta) = \psi_{\delta}(\bfx)v(\bfx)$. We have $\mrm{supp}(\chi_{\delta}v)\subset \Omega\setminus 
\mrB(O,2\delta)$, and $\mrm{supp}(V_{\delta})\subset \overline{\mrB(O,2)}$. Observe that $D(\mA^{0})\subset 
\mV^{2}_{\beta}(\Omega)$ so that, according to Proposition \ref{propoLaplaceBounded}, there holds $\Vert v\Vert_{\mV^{2}_{\beta}(\Omega)}\leq C \Vert \Delta v\Vert_{\mV^{0}_{\beta}(\Omega)}$. Here and in the sequel, $C>0$ denotes a constant independent of $\delta$ which can change from one line to another. As a consequence, we can write 
\begin{equation}\label{Estimate1}
\begin{array}{lcl}
\Vert\chi_{\delta}v \Vert_{\mV^{2}_{\beta,\delta}(\Omega)}
& \leq &  C\, \Vert v \Vert_{\mV^{2}_{\beta,\delta}(\Omega\setminus\overline{\mrB(O,2\delta)})} 
\leq C\,\Vert v \Vert_{\mV^{2}_{\beta}(\Omega\setminus\overline{\mrB(O,2\delta)})} \\[6pt]

& \leq & C\, \Vert v \Vert_{\mV^{2}_{\beta}(\Omega )}
\leq C\,\Vert \mathcal{A}_{\beta} v \Vert_{\mV^{0}_{\beta}(\Omega)}\leq C\,\Vert \mA^{0}v \Vert_{\mV^{0}_{\beta,\delta}(\Omega)}.
\end{array}
\end{equation}
In (\ref{Estimate1}), we used that $r\leq r+\delta\leq 2r$ in $\Omega\setminus\overline{\mrB(O,2\delta)}$ 
and that $\Vert v\Vert_{\mV^{0}_{\beta}(\Omega)}\leq \Vert v\Vert_{\mV^{0}_{\beta,\delta}(\Omega)}$ since $\beta\geq 0$. For $\vert \bfxi\vert \geq 2$, there holds $V_{\delta}(\bfxi) = 0$. Therefore, the usual elliptic \textit{a priori} estimates give $\Vert V_{\delta}\Vert_{\mathcal{V}^{2}_{\beta}(\R^{3}  )}\leq C \Vert \Delta V_{\delta}\Vert_{\mathcal{V}^{0}_{\beta}(\R^{3})}$. Making the change of variables $\bfxi = \bfx/\delta$, we deduce 
\begin{equation}\label{Estimate2}
\begin{array}{lcl}
\Vert \psi_{\delta}v\Vert_{\mV^{2}_{\beta,\delta}(\Om)}
& \leq & C \Vert \Delta (\psi_{\delta}v)\Vert_{\mV^{0}_{\beta,\delta}(\Om)}\\[6pt]

 & \leq & C\, \Vert \psi_{\delta} \Delta v\Vert_{\mV^{0}_{\beta,\delta}(\Om)} +
     C\, \delta^{-1}\Vert \nabla v\Vert_{\mV^{0}_{\beta,\delta}(\mathbb{Q}^{\delta})}+C\,\delta^{-2} \Vert v\Vert_{\mV^{0}_{\beta,\delta}(\mathbb{Q}^{\delta})} \\[6pt]

& \leq & C\, \Vert \mA^{0}v\Vert_{\mV^{0}_{\beta,\delta}(\Om)} + C\, \delta^{-1}\Vert \nabla v\Vert_{\mV^{0}_{\beta}(\mathbb{Q}^{\delta})}+C\,\delta^{-2} \Vert v\Vert_{\mV^{0}_{\beta}(\mathbb{Q}^{\delta})} \\[6pt]

& \leq & C\, \Vert \mA^{0}v\Vert_{\mV^{0}_{\beta,\delta}(\Om)} +
     C\, \Vert v\Vert_{\mV^{2}_{\beta}(\Omega)} \\[6pt]

 & \leq & C\, \Vert \mA^{0}v\Vert_{\mV^{0}_{\beta,\delta}(\Om)} +
     C\, \Vert \mathcal{A}_{\beta} v\Vert_{\mV^{0}_{\beta}(\Omega)}
\leq C\, \Vert \mA^{0}v\Vert_{\mV^{0}_{\beta,\delta}(\Om)}.
\end{array}
\end{equation}
There only remains to gather (\ref{Estimate1}) and  (\ref{Estimate2}) that hold
for any element $v\in D(\mA^{0})$ and any $\delta\in(0;\delta_0]$. This leads to the conclusion of the proof.
\end{proof}

\noindent Error estimate (\ref{error estimate result}) can also be formulated in $\mL^{2}$--norm. Indeed, observe that there exists a constant independent of $\delta$ such that $1\leq C/(\vert\bfx\vert+\delta)$ 
for all $\bfx\in\Omega$. This implies $\Vert v\Vert_{\Omega}\leq C^{2-\beta}
\Vert v\Vert_{\mV^{0}_{\beta-2,\delta}(\Omega)}\leq C^{2-\beta}\Vert v\Vert_{\mV^{2}_{\beta,\delta}(\Omega\setminus\Gamma^{\delta})}$ 
for all $v\in \mH^{2}(\Omega\setminus\Gamma^{\delta})$. From Proposition \ref{propoEstimResolvantes}, we deduce the following result.
\begin{corollary}\label{coroErrorEtsimate}
Assume that $\kappa_{\sigma}\neq -1$ and that Assumption \ref{assumption1} holds. Then for any $\eps\in(0;1)$, there exist constants $C_{\eps},\delta_0>0$ independent of $\delta$ such that 
\begin{equation}\label{error estimate result L2}
\sup_{f\in\mL^{2}(\Omega)\setminus\{0\}}\frac{\Vert (\mA^{0})^{-1} f- (\mA^{\delta})^{-1}f\Vert_{\Omega }}{
\Vert f\Vert_{\Om}}\leq C_{\eps}\,\delta^{3/2-\eps},\qquad\forall \delta\in(0;\delta_0].
\end{equation}
\end{corollary}
\begin{remark}
Notice that Proposition \ref{propoEstimResolvantes} and Corollary \ref{coroErrorEtsimate} prove that the function $v$ defined in (\ref{source term limit}) is a good approximation of $u^{\delta}$, the unique solution to Problem (\ref{PbAprioriEstimate}). Indeed, from (\ref{error estimate result}) with $\beta=1$ and from (\ref{error estimate result L2}), we can write, for all $\eps\in(0;1)$ and for $\delta$ small enough,
\[
\|u^{\delta}-v\|_{\mH^1_0(\Om)} \le C\,\delta\|f\|_{\Om}\qquad\mbox{ and }\qquad\|u^{\delta}-v\|_{\Om} \le C_{\eps}\,\delta^{3/2-\eps}\|f\|_{\Om}.
\]
\end{remark}
\noindent This ends the asymptotic analysis of the source term problem. In the sequel, we turn back to the initial spectral problem. To make the transition, observe that the result of Corollary \ref{coroErrorEtsimate} can be rephrased as $(\mA^{\delta})^{-1} = (\mA^{0})^{-1} + O(\delta^{3/2-\eps})$ considering $(\mA^{\delta})^{-1},(\mA^{0})^{-1}$ as operators mapping $\mL^{2}(\Omega)$ to $\mL^{2}(\Omega)$. Since these two operators are self-adjoint, a direct application of Theorem 4.10 of Chapter V of \cite{Kato95} yields that the spectra of $(\mA^{\delta})^{-1}$ and $(\mA^{0})^{-1}$ are closed to each other.
\begin{proposition}\label{SpectralConvergence}
Assume that $\kappa_{\sigma}\neq -1$ and that Assumption \ref{assumption1} holds. Then for any $\eps\in(0;1)$, there exist constants $C_{\eps},\delta_0>0$ independent of $\delta$ such that
\begin{equation}\label{distanceSpectra}
\sup_{\mu\in \mathfrak{S}(\mA^{0})}\mathop{\inf\phantom{p}}_{\lambda\in \mathfrak{S}(\mA^{\delta})}\Big\vert \frac{1}{\lambda}-\frac{1}{\mu}\Big\vert\; +\; 
\sup_{\lambda\in \mathfrak{S}(\mA^{\delta})}\mathop{\inf\phantom{p}}_{\mu\in \mathfrak{S}(\mA^{0})}\Big\vert \frac{1}{\lambda}-\frac{1}{\mu}\Big\vert\;\leq\; 
C_{\eps}\;\delta^{3/2-\eps},\qquad\forall \delta\in(0;\delta_0].
\end{equation}
\end{proposition}

\begin{remark}
The use of \cite[Thm 4.10, Chap V]{Kato95} (or Lemma \ref{EstimEigenVal} in Appendix) is crucial here in order to convert consistency estimates into spectral estimates. This result requires that the operator under consideration be at least normal. When the latter assumption is not satisfied, probably new significant difficulties arise.
\end{remark}

\section{Study of the positive spectrum}\label{StudyPositiveSpec}
\noindent In this section, we exploit Proposition \ref{SpectralConvergence} to prove our main result concerning the positive spectrum of $\mA^{\delta}$, namely, it converges to the spectrum of $\mA^{0}$.
\begin{theorem}\label{thmMajorPos}
Assume that $\kappa_{\sigma}\neq -1$ and that Assumption \ref{assumption1} holds. Let $n\in\N^{\ast}$ be a fixed number and let $\lambda_{n}^{\delta}>0$ (resp. $\mu_{n}>0$) refer to the corresponding eigenvalue of the operator $\mA^{\delta}$ (resp. $\mA^{0}$). Then there exist constants $C_{\eps},\delta_{0}>0$, depending on $n,\eps$ but independent of $\delta$, such that
\begin{equation}\label{resultPosSpectrum}
\vert \mu_{n} -\lambda_{n}^{\delta}\vert \leq C_{\eps}\,\delta^{3/2-\eps},\qquad\forall \delta \in (0;\delta_{0}]. 
\end{equation}
\end{theorem}
\begin{proof}
Pick $\eps\in(0;1)$. Let $\mu_n$, $n\in\N^{\ast}$, be an eigenvalue of multiplicity $\varkappa\ge1$ of the far field operator $\mA^0$. To set our ideas and without loss of generality, we assume that there holds 
$$
\mu_{n-1} < \mu_{n}=\cdots = \mu_{n+\varkappa-1}<\mu_{n+\varkappa}
$$
if $n>1$. Estimate (\ref{error estimate result}) establishes strong convergence of $(\mA^{\delta})^{-1}$ toward $(\mA^{0})^{-1}$. Therefore, we can apply Theorem 3.16 of Chapter IV of \cite{Kato95} which ensures that for $C>0$ large enough, for all $\delta\in(0;\delta_0]$ and all $i\in\{1,\dots,n\}$, the total multiplicity of the spectrum of $(\mA^{\delta})^{-1}$ in $\mathcal{O}_i^{\delta}:=[\mu_{i}^{-1}-C\,\delta^{3/2-\eps};\mu_{i}^{-1}+C\,\delta^{3/2-\eps}]$ is equal to $\varkappa_i$, the multiplicity of $\mu_{i}^{-1}$. We denote $\lambda^{\delta}_k\le\cdots\le \lambda^{\delta}_{k+\varkappa-1}$, for some $k\in\N^{\ast}$, the eigenvalues of $\mA^{\delta}$ whose inverses are located in $[\mu_{n}^{-1}-C\,\delta^{3/2-\eps};\mu_{n}^{-1}+C\,\delta^{3/2-\eps}]$. Now, we wish to prove that for all $\delta\in(0;\delta_0]$, we have $k=n$. For all $i<n$ and $\delta\in(0;\delta_0]$, we just showed that there are exactly $\varkappa_i$ eigenvalues of $(\mA^{\delta})^{-1}$ in $\mathcal{O}_i^{\delta}$. As a consequence, the relation $k\ge n$ is valid. Now, assume that $k>n$ for some $\delta\in(0;\delta_0]$. In this case, because of the relation on the multiplicity, there exists some $j\in\N^{\ast}$ such that $1/\lambda^{\delta}_j\notin\cup_{i=1}^{n-1}\mathcal{O}_i^{\delta}$. If the $C$ of the definition of the $\mathcal{O}_i^{\delta}$ is larger than the $C_{\eps}$ of estimate (\ref{distanceSpectra}), this is impossible. Therefore, we have $k=n$ and $\vert 1/\mu_{n} -1/\lambda_{n}^{\delta}\vert \leq C\,\delta^{3/2-\eps}$. To conclude, it remains to notice that the latter inequality is equivalent to (\ref{resultPosSpectrum}) because $\mu_{n}>0$.
\end{proof}
\begin{remark}
Observe that in (\ref{resultPosSpectrum}), the dependence of $C_{\eps}$ with respect to $n\in\N^{\ast}$ can be explicitly computed since there holds $C_{\eps}=O(\mu_n^2)$ when $n\to+\infty$.
\end{remark}

\section{Study of the negative spectrum}\label{negativeSpectrum}
\noindent The presence of negative eigenvalues in the spectrum of $\mA^{\delta}$ is due to the small piece 
of negative material in the domain $\Omega$. We will see that the behaviour of negative eigenvalues is driven by the near field structure of the operator $\mA^{\delta}$. More precisely, we will prove that the negative part of the spectrum of $\mA^{\delta}$ is asymptotically equivalent (in a precise meaning that we shall provide later) to the negative part of the spectrum of $\delta^{-2}\mrm{B}^{\infty}$. We shall proceed in two steps (so called inverse and direct reductions) following and adapting the approach considered in \cite{Naza93}.\\
\newline
Let us first introduce notation adapted to the study of the operator $\mA^{\delta}$ close to the 
small inclusion. Throughout this section, we shall make intensive use of the fast variable
$\bfxi = \bfx/\delta$. In this rescaled variable, the domain $\Omega$ becomes the 
$\delta$--dependent domain $\Xi^{\delta}:=\{\bfxi\in\R^3\,|\,\delta\bfxi\in\Om\}$ while the 
operator $\mA^{\delta}$ is changed into the operator $\mrm{B}^{\delta}:D(\mrm{B}^{\delta})\to 
\mrm{L}^{2}(\Xi^{\delta})$ defined by
$$
\begin{array}{|l}
\mrm{B}^{\delta} w^{\delta}  \;=\; -\div(\sigma^{\infty}\nabla w^{\delta})\\[5pt]
D(\mrm{B}^{\delta}) \;:=\; \{\,w^{\delta}\in \mH^{1}_{0}(\Xi^{\delta})\;\vert\;
\div(\sigma^{\infty}\nabla w^{\delta})\in\mL^{2}(\Xi^{\delta})\,\}.
\end{array}
$$
A simple calculus shows directly that $\mathfrak{S}(\mrm{B}^{\delta}) = \{\delta^{2}\lambda\;
\vert\;\lambda\in\mathfrak{S}(\mA^{\delta})\}$. Therefore, $\mathfrak{S}(\mrm{B}^{\delta})$ consists in the discrete set of ordered 
real values $\dots \leq \delta^{2}\lambda_{-2}^{\delta}\leq \delta^{2}\lambda_{-1}^{\delta}<0\leq 
\delta^{2}\lambda_{1}^{\delta}\leq \delta^{2}\lambda_{2}^{\delta}\leq \dots$. Here we will focus on the 
negative part of this set, as we want to show that $\mathfrak{S}(\mrm{B}^{\delta})\setminus\overline{\R_{+}}$ 
and $\mathfrak{S}(\mrm{B}^{\infty})\setminus\overline{\R_{+}}$ get close to each other as $\delta\to 0$. From the study of \S\ref{sectionSourceTerm}, we can already state the following preliminary result concerning the negative spectrum of $\mrm{A}^{\delta}$.
\begin{lemma}\label{EigenValuesBlowUp}
Assume that $\kappa_{\sigma}\neq -1$ and that Assumption \ref{assumption1} holds.
Then for any $\eps\in (0;1)$, there exist constants $C_{\eps},\delta_{0}>0$ 
independent of $\delta$ such that 
\begin{equation}\label{EstimEmptySpectrum}
\mathfrak{S}(\mA^{\delta})\cap (- C_{\eps}\,\delta^{-3/2+\eps};0) = \emptyset,\qquad \forall \delta\in(0;\delta_{0}].
\end{equation}
\end{lemma}
\begin{proof}
Proposition \ref{SpectralConvergence} ensures that 
$$
\sup_{\lambda\in \mathfrak{S}(\mA^{\delta})}\mathop{\inf\phantom{p}}_{\mu\in \mathfrak{S}(\mA^{0})}\Big\vert \frac{1}{\lambda}-\frac{1}{\mu}\Big\vert\;\leq\; 
C_{\eps}\;\delta^{3/2-\eps},\qquad\forall \delta\in(0;\delta_0].
$$
But we know that the spectrum of $\mA^{0}$ is strictly positive. Therefore, for all $\delta\in(0;\delta_0]$, we have $\sup_{\lambda\in \mathfrak{S}(\mA^{\delta})\setminus\overline{\R_{+}}} \vert \lambda\vert^{-1}\;\leq\; 
C_{\eps}\;\delta^{3/2-\eps}$. This implies (\ref{EstimEmptySpectrum}).
\end{proof}

\subsection{Inverse reduction}\label{sectionInverseReduction}

\noindent In this section, we establish that if $\mu_{-n}$ is an eigenvalue of $\mB^{\infty}$, then close to $\delta^{-2}\mu_{-n}$ there is an element of the negative part of the spectrum of $\mB^{\delta}$. The method is classical: from eigenpairs of the limit operator $\mB^{\infty}$, we construct approximation of eigenpairs of $\mB^{\delta}$. Then, we use the well-known lemma on ``near eigenvalues and eigenfunctions'' (see \cite{ViLy62}). Since we start from eigenpairs of the limit problem to build eigenpairs for the original problem, this approach is usually called \textit{inverse reduction}. We will see that the exponential decay of the eigenfunctions of $\mB^{\infty}$ established in \S\ref{ParNearFieldOperator} induces exponential convergence. 
\begin{lemma}\label{lemmaEstimateInverseNeg}
Assume that $\kappa_{\sigma}\neq -1$. For any $\mu\in \mathfrak{S}(\mB^{\infty})\setminus\overline{\R_{+}}$, 
there exist $C,\delta_{0}>0$ independent of $\delta$ such that 
\begin{equation}\label{estimates lemma inverse negative}
\inf_{w^{\delta}\in D(\mB^{\delta})\setminus\{0\}}\frac{\Vert \mB^{\delta} w^{\delta} - \mu w^{\delta} \Vert_{\Xi^{\delta}}}{\Vert w^{\delta}\Vert_{\Xi^{\delta}}}
\leq  C \delta\exp\left(-\frac{1}{2\delta}\sqrt{\frac{\vert\mu\vert}{\sigma_{+}}}\,\right),\qquad\forall \delta\in(0;\delta_{0}].
\end{equation}
\end{lemma}
\begin{proof}
Pick an arbitrary $\mu\in \mathfrak{S}(\mB^{\infty})\setminus\overline{\R_{+}}$ and consider a corresponding eigenfunction $w$ that satisfies $\Vert w\Vert_{\R^3} = 1$. Set $w^{\delta}\in D(\mB^{\delta})$ by $w^{\delta} = \psi^{\delta}w$, where $\psi^{\delta}$ is the function such that $\psi^{\delta}(\bfxi):= \psi(\delta\bfxi)$ (see (\ref{definition cut-off functions}) for the defintion of $\psi$). Let us use $w^{\delta}$ to prove (\ref{estimates lemma inverse negative}). We have 
$\div(\sigma^{\infty}\nabla w^{\delta}) =\psi^{\delta}\,\div(\sigma^{\infty}\nabla w)+\ 2\sigma_+\nabla 
\psi^{\delta}\cdot\nabla w+\sigma_{+}w\Delta\psi^{\delta}$. Since $\nabla\psi^{\delta}$, $\Delta\psi^{\delta}$ vanish 
outside $\mathbb{Q}^{1/\delta}=\{\bfxi\in\R^3\,|\,\delta^{-1}<|\bfxi|<2\delta^{-1}\}$, applying Proposition 
\ref{proposition decompo eigenfunction Outter}, we see that for $\gamma = \sqrt{\vert\mu\vert/\sigma_{+}}$, we have 
\[
\begin{array}{ll}
\|\mrB^{\delta}w^{\delta}-\mu\,w^{\delta}\|_{\Xi^{\delta}}^{2} & = 
\|2\sigma_+\nabla \psi^{\delta}\cdot\nabla w+\sigma_+ w\Delta \psi^{\delta}\|_{\Xi^{\delta}}^{2} \\[6pt]
& \le C\,(\delta^{2}\|\nabla w\|_{\mathbb{Q}^{1/\delta}}^{2}+\delta^{4}\|w \|_{\mathbb{Q}^{1/\delta}}^{2}) \\[6pt]
& \le  C\,\delta^{2}\exp(-\gamma/\delta)\dsp\int_{\R^{3}}(\vert w \vert^{2} + \vert \nabla w\vert^{2})\exp(\gamma\vert\bfxi\vert)\,d\bfxi .
\end{array}
\]
On the other hand, using the dominated convergence theorem, we see that $\|w^{\delta}\|_{\Xi^{\delta}}$ 
tends to $\|w\|_{\R^3}=1$ when $\delta\to0$. As a consequence, we have 
$\|w^{\delta}\|_{\Xi^{\delta}} \ge c>0$ for $\delta$ small enough, so that $\|\mB^{\delta}w^{\delta}-\mu w^{\delta}\|_{\Xi^{\delta}}/
\|w^{\delta}\|_{\Xi^{\delta}} = O(\delta\exp(-\gamma/2\delta))$.
\end{proof}

\noindent 
To obtain the next proposition, we use Lemma \ref{lemmaEstimateInverseNeg} and we apply the classical result on ``near eigenvalues and eigenfunctions'' (see Lemma \ref{EstimEigenVal} in appendix) to $\mrm{B}^{\delta}:D(\mrm{B}^{\delta})\to \mL^{2}(\Xi^{\delta})$. Observe that the latter operator, which is  self-adjoint, is indeed normal.

\begin{proposition}\label{thm approx eigen negative}
Assume that $\kappa_{\sigma}\neq -1$. For any $\mu\in \mathfrak{S}(\mB^{\infty})\setminus\overline{\R_{+}}$, 
there exists $C,\gamma,\delta_{0}>0$ independent of $\delta$ such that 
\begin{equation}\label{EstimatesNegativeEigenval}
\inf_{\lambda\in \mathfrak{S}(\mrm{A}^{\delta})}\vert \lambda - \delta^{-2}\mu\vert
\leq  C\,\exp(-\gamma/\delta),\qquad\forall \delta\in(0;\delta_{0}].
\end{equation}
\end{proposition}
\noindent With Lemma \ref{EigenValuesBlowUp}, we have seen that for all $\eps>0$, the set $\mathfrak{S}(\mA^{\delta})\cap (- C_{\eps}\,\delta^{-3/2+\eps};0)$ is empty for $\delta$ small enough. The next result indicates that, as $\delta\to0$, the negative eigenvalues of $\mA^{\delta}$ tend to $-\infty$ not too fast though.
\begin{corollary}\label{BoundednessEigenvalues}
Assume that $\kappa_{\sigma}\neq -1$. Let $n\in\N^{\ast}$ be a fixed number and let  $\lambda_{-n}^{\delta}<0$ refer to the corresponding eigenvalue of the operator $\mA^{\delta}$ as defined by Proposition \ref{spectrum delta}. Then, we 
have $\limsup_{\delta\to 0}\delta^{2}\vert\lambda_{-n}^{\delta}\vert <+\infty$.
\end{corollary}
\begin{proof}
Assume that there exist $n\in\N^{\ast}$ and a sequence $\delta_{k}\to 0$ such that 
$\lim_{k\to+\infty}\delta_{k}^{2}\lambda_{-n}^{\delta_{k}}= -\infty$. Since we have $\lambda_{-m}\le\lambda_{-n}$ for $m\ge n$, we conclude 
that $\lim_{k\to+\infty}\delta_{k}^{2}\lambda_{-m}^{\delta_{k}}= -\infty$ for all $m\ge n$. 
So only the sequences $(\delta_{k}^{2}\lambda_{-m}^{\delta_{k}})_{k\geq 0}$ with $m\in\{1,\dots,n-1\}$ may possibly remain bounded.\\
\newline
Now, consider $0> \mu^{\ell_1}>\dots > \mu^{\ell_n}$ $n$ distinct 
elements of $\mathfrak{S}(\mrm{B}^{\infty})\setminus\overline{\R_{+}}$. According to Proposition \ref{thm approx eigen negative}, there exist constants $C,\gamma,\delta_{0}>0$
independent of $\delta$ such that each interval $[\mu^{\ell_m}-C\,\delta^2\,\exp(-\gamma/\delta);\mu^{\ell_m}+C\,\delta^2\,\exp(-\gamma/\delta)]$, $m=1\dots n$, contains at least one element 
of the form $\delta^{2}\lambda^{\delta}_{j}$ for all $\delta\in(0;\delta_{0}]$. This imposes that we have $\limsup_{\delta\to 0}\delta^{2}\vert\lambda^{\delta}_{j}\vert<+\infty$
for at least $n$ distinct eigenvalues $\lambda_{j}^{\delta}\in\mathfrak{S}(\mrm{A}^{\delta})$. This is in contradiction with what precedes. Therefore, there holds $\limsup_{\delta\to 0}\delta^{2}\vert\lambda_{-n}^{\delta}\vert <+\infty$ for all $n\in\N^{\ast}$.
\end{proof}

\subsection{Direct reduction}\label{sectionDirectReduction}

\noindent From Proposition \ref{thm approx eigen negative}, we know that for all $n\in\N^{\ast}$, if $\mu_{-n}$ is an eigenvalue of $\mB^{\infty}$, then close to $\mu_{-n}$ there is a negative eigenvalue of $\mB^{\delta}$. In this section, we prove the converse assertion. We use the same technique as in the previous section: from eigenpairs of $\mB^{\delta}$, we build approximations of eigenpairs of the limit operator $\mB^{\infty}$. Then, we conclude thanks to the lemma on ``near eigenvalues and eigenfunctions''. This process is called the \textit{direct reduction} because we start from eigenpairs of the original problem to construct eigenpairs of the limit problem. As is often the case, the \textit{direct reduction} will be slightly more complicated than the \textit{inverse reduction}. We start by a preliminary lemma where we show a localization effect for the eigenfunctions of $\mB^{\delta}$ associated with the negative eigenvalues.

\begin{lemma}\label{LemmaestimLocalBounded}
Assume that $\kappa_{\sigma}\neq -1$ and that Assumption \ref{assumption1} holds. Let $n\in\N^{\ast}$ be a fixed number and let $w_{-n}^{\delta}$, such that $\Vert w_{-n}^{\delta}\Vert_{\Xi^{\delta}} = 1$, refer to an eigenfunction of $\mB^{\delta}$ corresponding to the negative eigenvalue $\delta^{2}\lambda_{-n}^{\delta}$. Then, we have
\begin{equation}\label{estimLocalBounded}
\limsup_{\delta\to 0}\int_{\Xi^{\delta}}(\delta^{2}\vert\lambda_{-n}^{\delta}\vert\,\vert w_{-n}^{\delta}(\bfxi)\vert^{2} +
\vert \nabla w_{-n}^{\delta}(\bfxi)\vert^{2})\exp\left(\vert\bfxi\vert\sqrt{\frac{\delta^{2}\vert\lambda_{-n}^{\delta}\vert}{\sigma_{+}}}\,\right) d\bfxi<+\infty.
\end{equation}
\end{lemma}
\begin{proof}
To simplify the notation, let us drop the index ``$_{-n}$''. We adopt here the same approach as in Proposition \ref{proposition decompo eigenfunction Outter}. Let us introduce the weight function $\mathscr{W}^{\delta}$ such that $\mathscr{W}^{\delta}(\bfxi) := \exp(\gamma^{\delta})$ for $\vert \bfxi\vert\leq 1$ and 
$\mathscr{W}^{\delta}(\bfxi) := \exp(\gamma^{\delta}\vert \bfxi\vert)$ for $\vert \bfxi\vert\geq 1$, with 
$\gamma^{\delta} = (\vert \delta^2\lambda^{\delta}\vert/\sigma_{+})^{1/2}/2$. Since 
$-\div(\sigma^{\infty}\nabla w^{\delta}) = \delta^{2}\lambda^{\delta}w^{\delta}$ in $\Xi^{\delta}$, a computation
nearly identical to that of the proof of Proposition \ref{proposition decompo eigenfunction Outter} 
shows that 
\begin{equation}\label{estimLocalBounded2}
\delta^{2}\vert \lambda^{\delta}\vert\|\mathscr{W}^{\delta}w^{\delta}\|_{\Xi^{\delta}}^{2}+
\sigma_{+}\|\mathscr{W}^{\delta}\nabla w^{\delta}\|_{\Xi^{\delta}}^{2} \le 
2\exp(2\gamma^{\delta})\,(\sigma_{+}+\vert\sigma_{-}\vert)\,\|\nabla 
w^{\delta}\|^2_{\Xi_{-}}.
\end{equation}
From Corollary \ref{BoundednessEigenvalues}, we know that the map $\delta\mapsto \delta^{2}\vert \lambda^{\delta}\vert$ (and so $\delta\mapsto \gamma^{\delta}$) is bounded as $\delta\to 0$. In addition, applying (\ref{estim1domain}) yields a 
constant  $C>0$ independent of $\delta,\lambda^{\delta}$, such that 
$\Vert \nabla w^{\delta}\Vert_{\Xi_{-}}\leq C(1+\delta^{2}\vert \lambda^{\delta}\vert)\Vert w^{\delta}\Vert_{\Xi_{-}}$. Since $\Vert w^{\delta}\Vert_{\Xi_{-}} \leq \Vert w^{\delta}\Vert_{\Xi^{\delta}} =  1$, we deduce that the right-hand side of the estimate (\ref{estimLocalBounded2}) remains bounded uniformly as $\delta\to 0$. This ends the proof.
\end{proof}

\noindent Note that if $w_{-n}^{\delta}$ is an eigenfunction of $\mrB^{\delta}$ with $\Vert w_{-n}^{\delta}\Vert_{\Xi^{\delta}} = 1$, then the function $v^{\delta}$ such that $v^{\delta}(\bfx) = \delta^{3/2}w_{-n}^{\delta}(\bfx/\delta)$ is an eigenfunction of $\mA^{\delta}$ with
$\Vert v^{\delta}\Vert_{\Xi^{\delta}} = 1$. Therefore, the previous lemma shows that eigenfunctions of $\mA^{\delta}$ associated with negative eigenvalues get more and more localized around the small negative inclusion as $\delta\to 0$. Now, we can construct quasi-eigenfunctions for $\mB^{\infty}$ from eigenfunctions of the operator $\mB^{\delta}$.

\begin{lemma}\label{lemma estimations inverse negative}
Assume that $\kappa_{\sigma}\neq -1$ and that Assumption \ref{assumption1} holds. Let $n\in\N^{\ast}$ be a fixed number and let $\delta^{2}\lambda_{-n}^{\delta}<0$ refer to the corresponding eigenvalue of the operator $\mB^{\delta}$. Then there exist constants $C,\gamma,\delta_{0}>0$, depending on $n$ but independent of $\delta$, such that 
\begin{equation}\label{estimates lemma direct negative}
\inf_{w\in D(\mrm{B}^{\infty})\setminus\{0\}}
\frac{\|\mB^{\infty} w - \delta^{2}\lambda_{-n}^{\delta}w\|_{\R^3}}{\Vert w\Vert_{\R^{3}}} \le C\,\delta\exp(-\gamma/\delta),\qquad\forall \delta\in(0;\delta_{0}].
\end{equation}
\end{lemma}
\begin{proof}
Once again, we omit the index ``$_{-n}$''. Let us consider an eigenfunction $w^{\delta}$ of $\mrB^{\delta}$ associated with the eigenvalue $\delta^{2}\lambda^{\delta}$. We assume that there holds $\Vert w^{\delta}\Vert_{\Xi^{\delta}} = 1$. Set $w := \psi^{\delta}\,w^{\delta}$ where $\psi^{\delta}$ is such that $\psi^{\delta}(\bfxi):=\psi(\delta\bfxi)$ ($\psi$ is the cut-off function defined in (\ref{definition cut-off functions})). Let us show first that  
\begin{equation}\label{LimitNorm}
\lim_{\delta\to 0}\Vert w\Vert_{\R^{3}} = 1.
\end{equation}
Introducing $\chi^{\delta} = 1-\psi^{\delta}$, we can write $\|w\|_{\R^3} = \|\psi^{\delta}w^{\delta}\|_{\R^3} \ge \| w^{\delta}\|_{\Xi^{\delta}} - \| \chi^{\delta} w^{\delta}\|_{\Xi^{\delta}}$. According to Lemma \ref{EigenValuesBlowUp}, there exists a constant 
$C>0$ independent of $\delta$ such that $\vert \lambda^{\delta}\vert\geq C\,\delta^{-1}$ 
so that $\delta^{2}\vert \lambda^{\delta}\vert\geq C\,\delta$ for $\delta$ small enough.
In addition, observe that $\vert\bfxi\vert\geq 1/\delta$ for $\bfxi\in\mrm{supp}(\chi^{\delta})$. 
As a consequence, applying Lemma \ref{LemmaestimLocalBounded} above, we can write
$$
\Vert \chi^{\delta} w^{\delta}\Vert_{\Xi^{\delta}}^{2}\leq 
\frac{\exp\Big(-\sqrt{\vert\lambda^{\delta}\vert/\sigma_{+}}\, \Big)}{\delta^{2}
\vert\lambda^{\delta}\vert}\;\int_{\Xi^{\delta}}\delta^{2}\vert\lambda^{\delta}\vert\,\vert w^{\delta}(\bfxi)\vert^{2}
\exp\Big(\vert\bfxi\vert\sqrt{\frac{\delta^{2}\vert\lambda^{\delta}\vert}{\sigma_{+}}} \Big)\,d\bfxi\leq 
C_{1}\delta^{-1}\exp(-C_{2}\delta^{-1/2}).
$$
This implies $\lim_{\delta\to 0}\Vert \chi^{\delta} w^{\delta}\Vert_{\Xi^{\delta}} = 0$ and establishes (\ref{LimitNorm}). Now, we estimate $\|\mB^{\infty} w-\delta^{2}\lambda w\|_{\R^3}$. We compute
$\div(\sigma^{\infty}\nabla w) =\psi^{\delta}\,\div(\sigma^{\infty}\nabla w^{\delta})+ 
2\sigma_+\nabla \psi^{\delta}\cdot\nabla w^{\delta}+\sigma_+w^{\delta}\Delta\psi^{\delta}$.
Observing that $|\nabla\psi^{\delta}|\le C\,\delta$, $|\Delta\psi^{\delta}|\le C\,\delta^{2}$ are non null 
only on $\mathbb{Q}^{1/\delta}=\{\bfxi\in\R^3\,|\,\delta^{-1}<|\bfxi|<2\delta^{-2}\}$ and recalling that 
$\div(\sigma^{\infty}\nabla w^{\delta}) = \delta^{2}\lambda^{\delta}w^{\delta}$, we conclude that there exists a 
constant $C>0$ independent of $\delta$ such that 
\begin{equation}\label{EstimDistSpec2}
\begin{array}{lcl}
\|\mB^{\infty} w-\delta^{2}\lambda^{\delta}w\|_{\R^3} & = & 
\|2\sigma_+\nabla \psi^{\delta}\cdot\nabla w^{\delta}+\sigma_+w^{\delta}\Delta\psi^{\delta}\|_{\Xi^{\delta}}\\[0.4em]
& \le & C\,(\delta\|\nabla w^{\delta}\|_{\mathbb{Q}^{1/\delta}}+\delta^{2}\|w^{\delta}\|_{\mathbb{Q}^{1/\delta}})
\end{array}
\end{equation}
Using Lemma \ref{LemmaestimLocalBounded} and working as in the beginning of this proof, we can show that $\limsup_{\delta\to 0}\Vert w^{\delta}\Vert_{\Xi^{\delta}} + 
\Vert \nabla w^{\delta}\Vert_{\Xi^{\delta}}<+\infty$. This remark, combined with 
the above inequality shows that $\|\mB^{\infty} w-\delta^{2}\lambda^{\delta}w\|_{\R^3} 
= O(\delta)$. Since $\mrm{B}^{\infty}$ is self-adjoint, we can apply the lemma on ``near eigenvalues and eigenfunctions'' (see Lemma \ref{EstimEigenVal} in appendix) to obtain
\begin{equation}\label{EstimDistSpec}
\inf_{\mu\in \mathfrak{S}(\mrm{B}^{\infty})}\vert\mu - \delta^{2}\lambda^{\delta}\vert = O(\delta)
\end{equation}
We deduce in particular that there exists $\delta_{0}>0$ independent of $\delta$ such that
\begin{equation}\label{EstimDistSpec3}
\delta^{2}\vert \lambda^{\delta}\vert\geq \vert \mu_{-1}\vert/2,\qquad\forall \delta\in(0;\delta_{0}].
\end{equation}
Now, let us come back to (\ref{EstimDistSpec2}). Since $\delta^{-1}\leq \vert\bfxi\vert\leq 2\delta^{-1}$ on $\mathbb{Q}^{1/\delta}$,
making use of (\ref{EstimDistSpec3}) combined with Lemma \ref{LemmaestimLocalBounded}, we deduce that there exists 
a constant $C>0$ independent of $\delta$ such that 
$\|\mB^{\infty} w^{\delta}-\delta^{2}\lambda_{-n}^{\delta}w^{\delta}\|_{\R^3}\leq C\delta\exp(-\delta^{-1}\sqrt{\vert\mu_{-1}\vert/(2\sigma_{+})})$. 
This, together with (\ref{LimitNorm}) leads to the conclusion of the proof.
\end{proof}

\noindent 
Apply one last time Lemma \ref{EstimEigenVal} in appendix to obtain the 
\begin{proposition}\label{propoSecondeEstimNeg}
Assume that $\kappa_{\sigma}\neq -1$ and that Assumption \ref{assumption1} holds. Let $n\in\N^{\ast}$ be a fixed number and let  $\lambda_{-n}^{\delta}<0$ refer to the corresponding eigenvalue of the operator $\mA^{\delta}$. Then there exist constants $C,\gamma,\delta_{0}>0$, depending on $n$ but independent of $\delta$, such that 
$$
\inf_{\mu\in \mathfrak{S}(\mrm{B}^{\infty})}\vert \lambda_{-n}^{\delta} - \delta^{-2}\mu\vert\leq C\,\exp(-\gamma/\delta)
,\qquad\forall \delta\in(0;\delta_{0}].
$$
\end{proposition}

\subsection{Conclusion}
\noindent In the next theorem, we state the main result of the paper concerning the negative spectrum of the original operator $\mA^{\delta}$.
\begin{theorem}\label{NegSpectConv}
Assume that $\kappa_{\sigma}\neq -1$ and that Assumption \ref{assumption1} holds. Let $n\in\N^{\ast}$ be a fixed number and let  $\lambda_{-n}^{\delta}<0$ (resp. $\mu_{-n}<0$) refer to the corresponding eigenvalue of the operator $\mA^{\delta}$ (resp. $\mB^{\infty}$). Then there exist constants 
$C,\gamma,\delta_{0}>0$, depending on $n$ but independent of $\delta$, such that 
$$
\vert \lambda_{-n}^{\delta} - \delta^{-2}\mu_{-n}\vert\leq C\exp(-\gamma/\delta)
,\qquad\forall \delta\in(0;\delta_{0}].
$$
\end{theorem}

\begin{proof} Pick some arbitrary $n\in\N^{\ast}$. Corollary \ref{BoundednessEigenvalues} ensures that the map $\delta\mapsto\delta^2\lambda_{-n}^{\delta}$ remains bounded on $(0;\delta_0]$ for some $\delta_0>0$. From Proposition \ref{propoSecondeEstimNeg}, we deduce that $\delta\mapsto\delta^2\lambda_{-n}^{\delta}$ is valued in 
$$
\dsp\bigcup_{i=1}^{I_{n}}\ \Big[\mu_{-i}-C\delta^{2}\exp(-\gamma/\delta)\,;\,\mu_{-i}+C\delta^{2}\exp(-\gamma/\delta)\Big],
$$
for some $I_{n}<+\infty$. Therefore, to prove the result of this theorem, it is sufficient to show that for $\delta$ small enough, $\delta\mapsto\delta^2\lambda_{-n}^{\delta}$ does not meet the interval $[\mu_{-i}-C\delta^{2}\exp(-\gamma/\delta);\mu_{-i}+C\delta^{2}\exp(-\gamma/\delta)]$ for $i\ne n$. Let us proceed by contradiction assuming that there exists a subsequence $(\delta_{k})$ such that $\lim_{k\to+\infty}\delta_{k}=0$ and $\lim_{k\to+\infty}\delta_{k}^{2}\lambda_{-n}^{\delta_{k}} =\mu_{\star}\in\mathfrak{S}(\mB^{\infty})\setminus\overline{\R_{+}}$, with $\mu_{\star}\neq \mu_{-n}$. Choosing $n$ closer to $1$ if necessary, one may assume that there holds  $\lim_{\delta\to 0}\delta^{2}\lambda_{-p}^{\delta} = \mu_{-p}$ for all $p=1,2,\dots ,n-1$ when $n\ge2$. In other words, we consider the smallest $n\in\N^{\ast}$ such that the set $\{\delta^2\lambda_{-n}^{\delta}\}_{\delta\in(0;\delta_0]}$ has at least two distinct points of accumulation.\\
\newline
\noindent\textbf{Notation.}
In the remainder of this proof, we will always consider values of $\delta\in 
\{\delta_{k}\}_{k\geq 0}$. To simplify notations, we will drop the subscript ``$k$'' in ``$\delta_{k}$''. This should not bring confusion. On the other hand, in the sequel, the sequence  $(w^{\delta}_{j})_{j\in\mathbb{Z}}$ will refer to an orthonormal Hilbert basis of $\mL^2(\Xi^{\delta})$, where $w^{\delta}_{j}$ is an eigenfunction of $\mrB^{\delta}$ corresponding to $\delta^{2}\lambda_{j}^{\delta}$.\\ 
\newline
\noindent \textbf{Step 1.}
Assume first that $\mu_{\star}>\mu_{-n}$. Let $\mathscr{E}^{\infty}$ refer to the eigenspace of $\mrB^{\infty}$ associated with the eigenvalue  $\mu_{\star}$. This is a finite dimensional space because $\mu_{\star}$ does not belong to the essential spectrum of $\mrB^{\infty}$. We introduce $\mathbb{P}^{\infty}: \mrm{L}^{2}(\R^{3})\to \mrm{L}^{2}(\R^{3})$ the continuous linear map such that $\mrm{Id} - \mathbb{P}^{\infty}$ is the spectral projector of $\mrB^{\infty}$ onto $\mathscr{E}^{\infty}$. On the other hand, we denote $J := \{j\in\N^{\ast}\,
\vert\,\mu_{-j} = \mu_{\star}\}\cup\{n\}$ and we set $\mathscr{F}^{\delta} := \mrm{span}\{\psi^{\delta}w_{j}^{\delta}\,\vert\, j\in J  \}\subset D(\mrm{B}^{\infty})$. Here and in the sequel, $\psi^{\delta}$ and $\chi^{\delta} = 1 - \psi^{\delta}$ are the cut-off functions defined in the proof of Lemma \ref{lemma estimations inverse negative}. Since $\mu_{\star}> \mu_{-n}$, we know that for all $\delta\in(0;\delta_0]$, there holds $\mrm{card}(J)= \mrm{dim}\,\mathscr{E}^{\infty}+1$, and that $\delta^{2}\lambda_{-j}^{\delta}\to\mu_{\star}$ for all $j\in J$ (here, we use the assumption that $n\in\N^{\ast}$ is the smallest index such that $\{\delta^2\lambda_{-n}^{\delta}\}_{\delta\in(0;\delta_0]}$ has at least two distinct points of accumulation). Moreover, remembering that $(w^{\delta}_{-j})_{j\in\mathbb{Z}}$ is an orthonormal family and observing that $\lim_{\delta\to 0}\Vert \chi^{\delta} w^{\delta}_{-j}\Vert_{\R^{3}} = 0$ (use the same arguments as in the proof of Lemma \ref{lemma estimations inverse negative}), one can show that  $(\psi^{\delta}w^{\delta}_{-j})_{j\in J}$ is linearly independent for $\delta$ small enough. This implies $\mrm{dim}\,\mathscr{F}^{\delta} > \mrm{dim}\,\mathscr{E}^{\infty}$. 
As a consequence, according to \cite[Thm 1.1]{GoKr60}, there exists $v^{\delta}\in \mathscr{F}^{\delta}$ such that 
$\Vert v^{\delta}\Vert_{\R^{3}} = 1$ and $\mrm{dist}(v^{\delta},\mathscr{E}^{\infty} ):=
\inf\{\Vert v^{\delta} - w\Vert_{\R^{3}}\;\vert\; w\in \mathscr{E}^{\infty}\} = 1$,
which can be rewritten $\Vert \mathbb{P}^{\infty} v^{\delta}\Vert_{\R^{3}} = 1$, for $\delta>0$ small enough. Using again that $\lim_{\delta\to 0}\Vert \chi^{\delta} w^{\delta}_{-j}\Vert_{\R^{3}} = 0$, we see that in the decomposition 
$$
v^{\delta} = \sum_{j\in J}
\alpha_{j}^{\delta}\,\psi^{\delta}w_{-j}^{\delta},
$$ 
where $\alpha_{j}^{\delta}\in \mathbb{C}$, there holds $ \sum_{j\in J}\vert\alpha_{j}^{\delta}\vert^{2}\to1$ when $\delta\to0$. Denoting $C = (\,\mrm{dist}(\mu_{\star}, \mathfrak{S}(\mrB^{\infty})\setminus\{\mu_{\star}\})\,)^{-2}$, we can write
$$
\begin{array}{lcl}
1 = \Vert \mathbb{P}^{\infty} v^{\delta}\Vert_{\R^{3}}^{2} & \leq & C\,\Vert (\mrm{B}^{\infty}- \mu_{\star}\mrm{Id})\mathbb{P}^{\infty} v^{\delta}\Vert_{\R^{3}}^{2}\\[5pt]

& \leq & C\,\Vert (\mrm{B}^{\infty}- \mu_{\star}\mrm{Id}) v^{\delta}\Vert_{\R^{3}}^{2}\\[5pt]

& \leq & C'\, \sum_{j\in J}\vert \alpha_{j}^{\delta}\vert^{2}\, 
\Vert (\mrm{B}^{\infty} - \mu_{\star}\mrm{Id}) \psi^{\delta}w_{-j}^{\delta}\Vert_{\R^{3}}^{2}  \\[5pt]

& \leq & C'\, \sum_{j\in J}\vert \alpha_{j}^{\delta}\vert^{2}\,  \big(\,\Vert (\mrm{B}^{\infty} - \delta^{2}\lambda_{-j}^{\delta}\mrm{Id}) 
\psi^{\delta}w_{-j}^{\delta}\Vert_{\R^{3}}^{2}  + \vert \mu_{\star}-\delta^{2}\lambda_{-j}^{\delta}\vert^{2}\Vert\psi^{\delta}w_{-j}^{\delta}\Vert_{\R^{3}}^{2}\,\big).

\end{array}
$$
We know, by assumption, that $\lim_{\delta\to 0}\vert \mu_{\star}-\delta^{2}\lambda_{-j}^{\delta}\vert =  0$ for each $j\in J$. Moreover, we have shown in the proof of Lemma \ref{lemma estimations inverse negative} that 
$\lim_{\delta\to 0}\Vert (\mrm{B}^{\infty} - \delta^{2}\lambda_{-j}^{\delta}\mrm{Id}) \psi^{\delta}w_{-j}^{\delta}\Vert_{\R^{3}}=0$. 
As a consequence, the inequality written above leads to a contradiction, which concludes the first step of the proof.\\
\newline
\textbf{Step 2.} The only remaining possibility is that $\mu_{\star}<\mu_{-n}$. But recall that we have $\delta^{2}\lambda_{-j}^{\delta}\le \delta^{2}\lambda_{-n}^{\delta}$ for all $j\ge n$. On the other hand, there holds $\lim_{\delta\to 0}\delta^{2}\lambda_{-p}^{\delta} = \mu_{-p}$ for all $p=1,2,\dots ,n-1$ (at least, when $n\ge2$). Therefore, $\mu_{\star}<\mu_{-n}$ implies that for $\delta_0$ small enough, we have $\{\delta^2\lambda_{-n}^{\delta}\}_{\delta\in(0;\delta_0]}\cap [\mu_{-n}-\delta;\mu_{-n}+\delta]=\emptyset$. This is in contradiction with Proposition \ref{thm approx eigen negative}. We conclude that the case $\mu_{\star}<\mu_{-n}$ is not possible either, which finishes the proof. 
 
\end{proof}

\section{Numerical experiments in 2D}\label{Numerics}

\noindent \begin{wrapfigure}{r}{0.6\textwidth}
\centering\includegraphics{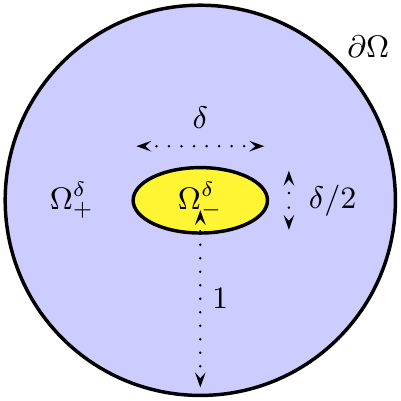}\vspace{-0.2cm}
\caption{Geometry of the domains.\label{geom part}}
\end{wrapfigure}
\noindent In this section, we approximate numerically the spectrum of Problem (\ref{ExPb}) set in a 2D domain. We consider a 2D configuration for computational 
reasons, the number of degrees of freedom being too large in 3D. Although the analysis is slightly more involved, we can prove that all the results we 
have established in 3D for the spectrum of $\mA^{\delta}$ also hold in 2D. We shall keep the same notations. First, we detail the parameters used for the 
numerical experiments. Let $\Om\subset\R^2$ be the disk of radius one centered at $O$. Define the ellipse $\Xi_-:=\{(x,y)\in\R^2\,|\,(x/a)^2+(y/b)^2<1\}$ 
with $a=1/2$ and $b=1/4$. For $\delta\in(0;1]$, set $\Om_{-}^{\delta}:=\delta\,\Xi_-$ and $\Om_{+}^{\delta}  := \Om\setminus\overline{\Om}\!\,_{-}^{\delta}$ (see Figure 
\ref{geom part}). Introduce the function $\sigma^{\delta}:\Omega\to \R$ such that $\sigma^{\delta} = \sigma_{\pm}$ in $\Omega_{\pm}^{\delta}$, 
where $\sigma_{+}=1$ and where $\sigma_{-}<0$ is a constant. We are interested in the 2D version of the eigenvalue problem (\ref{ExPb}) whose variational formulation writes
\begin{equation}\label{ExPb canon}
\begin{array}{|l}
\mbox{Find }(\lambda^{\delta},u^{\delta})\in\Cplx\times(\mH^{1}_{0}(\Om)\setminus\{0\})\mbox{ such that}\\[4pt]
(\sigma^{\delta}\nabla u^{\delta},\nabla v)_{\Om} = \lambda^{\delta}(u^{\delta},v)_{\Om},\qquad \forall v\in\mH^{1}_{0}(\Om).
\end{array}
\end{equation}

\noindent Let us explain how to choose $\sigma_{-}$ so that Assumption \ref{assumption1} holds. Using \cite[Prop 8]{KhPS07}, when the contrast 
$\kappa_{\sigma}=\sigma_{-}/\sigma_{+}<0$ verifies $\kappa_{\sigma}\ne-1$, we can prove that the near field operator $\mB^{\infty}$ defined in (\ref{NFOp}) 
is injective  if and only if \footnote{Notice that when $\Xi_-$ is a disk, \textit{i.e.} when $a=b$, the set $\mathscr{S}$ reduces to $\{-1\}$.} 
\[
\kappa_{\sigma}\notin \mathscr{S}:=\{\eta_k\}_{k=1}^{+\infty}\cup\{1/\eta_k\}_{k=1}^{+\infty},\qquad\mbox{where}\quad\eta_k:=\dsp\frac{|a-b|^k+|a+b|^k}{|a-b|^k-|a+b|^k}.
\]
It is straightforward to check that $(\eta_k)_{k}$ is a sequence of numbers smaller than $-1$, strictly increasing, which accumulates in $-1$. 
Moreover, for the selected geometry, we have $\eta_1=-a/b=-2$. In order to satisfy Assumption \ref{assumption1}, we will impose $\sigma_{-}=-2.5$ so 
that $\kappa_{\sigma}$ avoids the critical set $\mathscr{S}$. Now, we discretize Problem (\ref{ExPb canon}). Let us consider $\Om_h$, $\Om_{\pm, h}^{\delta}$ polygonal approximations of the domains $\Om$, $\Om_{\pm}^{\delta}$ satisfying $\overline{\Om_h}=\overline{\Om_{+, h}^{\delta}}\cup\overline{\Om_{-, h}^{\delta}}$. We denote $\sigma^{\delta}_h:\Om_h\to \R$ the function such that $\sigma^{\delta}_h = \sigma_{\pm}$ in $\Omega_{\pm,h}^{\delta}$. Introduce $(\mathcal{T}^{\delta}_h)_h$ a shape regular family of triangulations of $\overline{\Omega_h}$, made of triangles. Here, $h$ refer to the mesh size. Moreover, we assume that, for 
any triangle $\tau$, one has either $\tau\subset \overline{\Omega^{\delta}_{+,h}}$ or $\tau\subset \overline{\Omega^{\delta}_{-,h}}$. Define the family of finite 
element spaces
\[
\mV^{\delta}_h := \left\{v\in \mH_{0}^{1}(\Om^{\delta}_h)\mbox{ such that }v|_{\tau}\in \mathbb{P}_1(\tau)\mbox{ for all }\tau\in\mathcal{T}^{\delta}_h\right\},
\]
where $\mathbb{P}_1(\tau)$ is the space of polynomials of degree at most $1$ on the triangle $\tau$. Let us consider the problem
\begin{equation}\label{Pb h}
\begin{array}{|l}
\dsp{ \textrm{Find}\;(\lambda^{\delta}_h,u^{\delta}_h)\in\Cplx\times(\mV^{\delta}_h\setminus\{0\})\mbox{ such that}\quad}\\[6pt]
(\sigma_h^{\delta} \nabla u^{\delta}_h,\nabla v^{\delta}_h)_{\Om^{\delta}_h} = \lambda^{\delta}_h(u^{\delta}_h,v^{\delta}_h)_{\Om^{\delta}_h},\quad
\forall v^{\delta}_h\in\mV^{\delta}_h.
\end{array}
\end{equation}
For details concerning the discretization process, we refer the reader to \cite{BoCZ10,NiVe11,ChCiAc}. Nevertheless, we should emphasize that in these three papers, only the source term problem 
associated with (\ref{ExPb canon}) is considered. To this day, it seems that there exists no proof of convergence (as the mesh size $h$ tends to zero) of the standard $\mathbb{P}_1$ Lagrange finite element method to approximate the spectral Problem (\ref{ExPb canon}). It stems from the fact that the sign-changing of $\sigma^{\delta}$ prevents the use of usual approaches. However, from a practical point of view, for a given $\delta$, we observe no particular phenomena due to the presence of negative material and the eigenvalues look correctly approximated as $h\to0$ (see Table \ref{Table1}). For the computations, we use the \textit{FreeFem++}\footnote{\textit{FreeFem++}, \url{http://www.freefem.org/ff++/}.} software while we display the results with \textit{Matlab}\footnote{\textit{Matlab}, \url{http://www.mathworks.se/}.} and \textit{Paraview}\footnote{\textit{Paraview}, 
\url{http://www.paraview.org/}.}.
\begin{table}[!ht]
\centering
\renewcommand{\arraystretch}{1.5}\begin{tabular}{|c|c|c|}
\hline
mesh size $h$ & $\lambda_{h,+1}^{\delta}$ & $\lambda_{h,-1}^{\delta}$ \\
\hline
0.0856833 & 5.90764 & -3.33104 \\
\hline
0.041797  & 5.85259 & -4.47873 \\
\hline
0.0214048 & 5.8386 & -4.80557 \\
\hline
0.00932046 & 5.83554 & -4.87614 \\
\hline
0.0061639 & 5.83511 & -4.88572\\
\hline
\end{tabular}
\caption{First positive and negative eigenvalues for several meshes with $\delta = 0.5$. This suggests a convergence of eigenvalues in $O(h^{2})$ toward their limit, which is in good agreement with the rate of convergence 
predicted by the Babuska-Osborn theory \cite{Babuska1991641} for the case of $\mathbb{P}_1$ discretization of eigenvalue problems associated to compact self-adjoint operators.\label{Table1}}
\end{table}

\noindent In the next numerical experiments, we shall fix the mesh size $h$ so that the number of triangles inside the small inclusion $\Omega^{\delta}_{-,h}$ remains approximately constant. As a consequence, when $\delta$ goes to zero, the mesh contains more and more triangles. We make this choice so as to capture 
the localized eigenfunctions associated with the negative eigenvalues (see Figure \ref{first eigenfunctions}).\\
\newline 
In Figure \ref{PositiveSpectrum}, we display the positive  eigenvalues of smallest modulus of Problem (\ref{Pb h}) with respect to $-\mrm{log_{10}}\,\delta$. 
The dotted lines represent the approximation of the eigenvalues of smallest modulus of the limit operator $\mA^0$ defined in (\ref{FFOp}). In other words, 
these dotted lines correspond to the spectrum of the problem without the inclusion of negative material. We observe that $\spec(\mA^{\delta})$ seems to 
converge to $\spec(\mA^{0})$ when the inclusion shrinks. This is in accordance with the analog of Theorem \ref{thmMajorPos} in 2D.\\
\newline
In Figure \ref{NegativeSpectrum}, we present the behaviour of the negative eigenvalues of smallest modulus of Problem (\ref{Pb h}) with respect to 
$-\mrm{log_{10}}\,\delta$. The numerical experiment suggests that the negative eigenvalues of the operator $\mA^{\delta}$ tend to $-\infty$ when $\delta$ 
goes to zero. This was established in 3D in Lemma \ref{EigenValuesBlowUp}. Figure \ref{NegativeSpectrumRegression} confirms and clarifies this result: 
it indicates that the negative eigenvalues of Problem (\ref{Pb h}) behave like $\delta^{-2}\mu$, for some constant $\mu<0$, when $\delta\to0$. This is coherent 
with the 2D version of Theorem \ref{NegSpectConv}.\\
\newline
Eventually, on Figure \ref{first eigenfunctions}, we display the eigenfunctions associated with the negative and positive eigenvalues of smallest modulus for 
$\delta=0.5$ and for $\delta=0.05$. The eigenfunction associated with the smallest positive eigenvalue appears independent of the size of the small inclusion. 
On the contrary, we can observe the localization effect for the eigenfunction associated with the larger negative eigenvalue: when $\delta$ tends to zero, 
this eigenfunction is more and more concentrated around the inclusion of negative material.

\begin{figure}[!ht]
\centering
\includegraphics[scale=0.65]{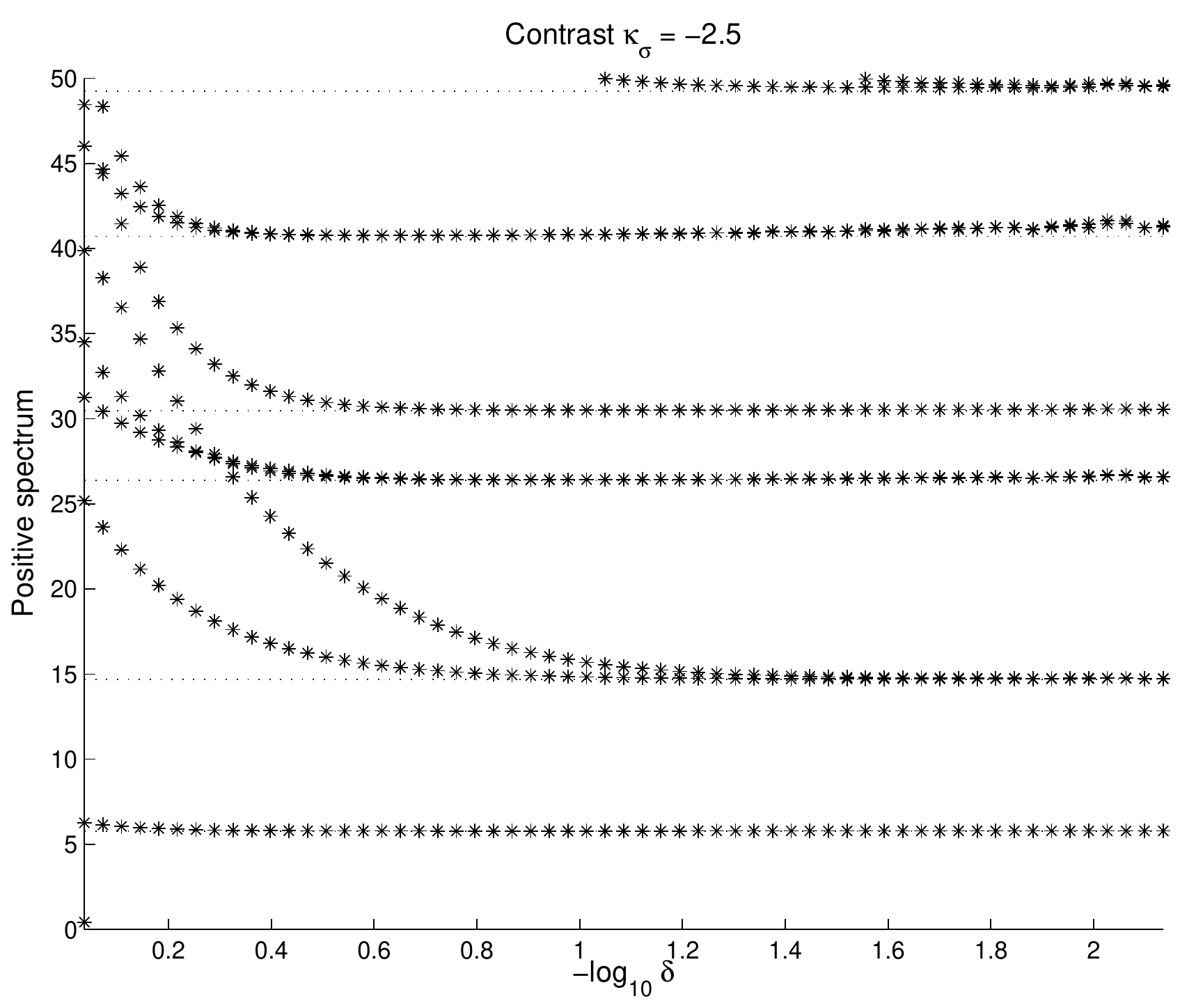}
\caption{For a given $\delta\in(0;1]$, we approximate the eigenvalues of smallest modulus of the operator $\mA^{\delta}$. 
Then, we make $\delta$ tend to zero. The figure represents the approximation of the positive spectrum of $\mA^{\delta}$ with 
respect to $-\mrm{log_{10}}\,\delta$. The dotted lines correspond to the approximation of the eigenvalues of smallest modulus of the limit operator $\mA^0$.}
\label{PositiveSpectrum}
\end{figure}

\begin{figure}[!ht]
\centering
\includegraphics[scale=0.65]{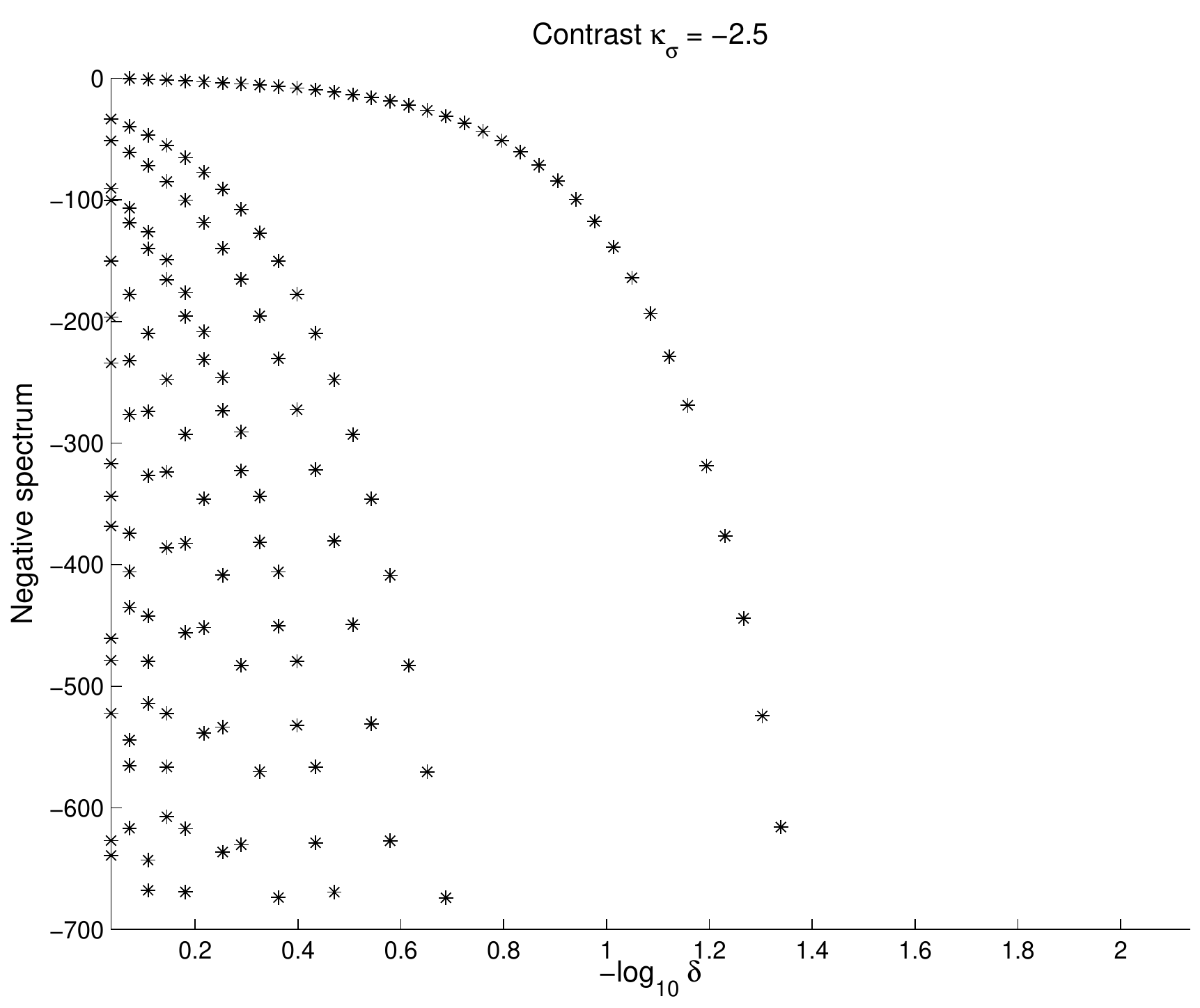}
\caption{For a given $\delta\in(0,1]$, we approximate the eigenvalues of smallest modulus of the operator $\mA^{\delta}$. 
Then, we make $\delta$ tend to zero. The figure represents the approximation of the negative spectrum of $\mA^{\delta}$ with respect to $-\mrm{log_{10}}\,\delta$.}
\label{NegativeSpectrum}
\end{figure}

\begin{figure}[!ht]
\centering
\includegraphics[scale=0.65]{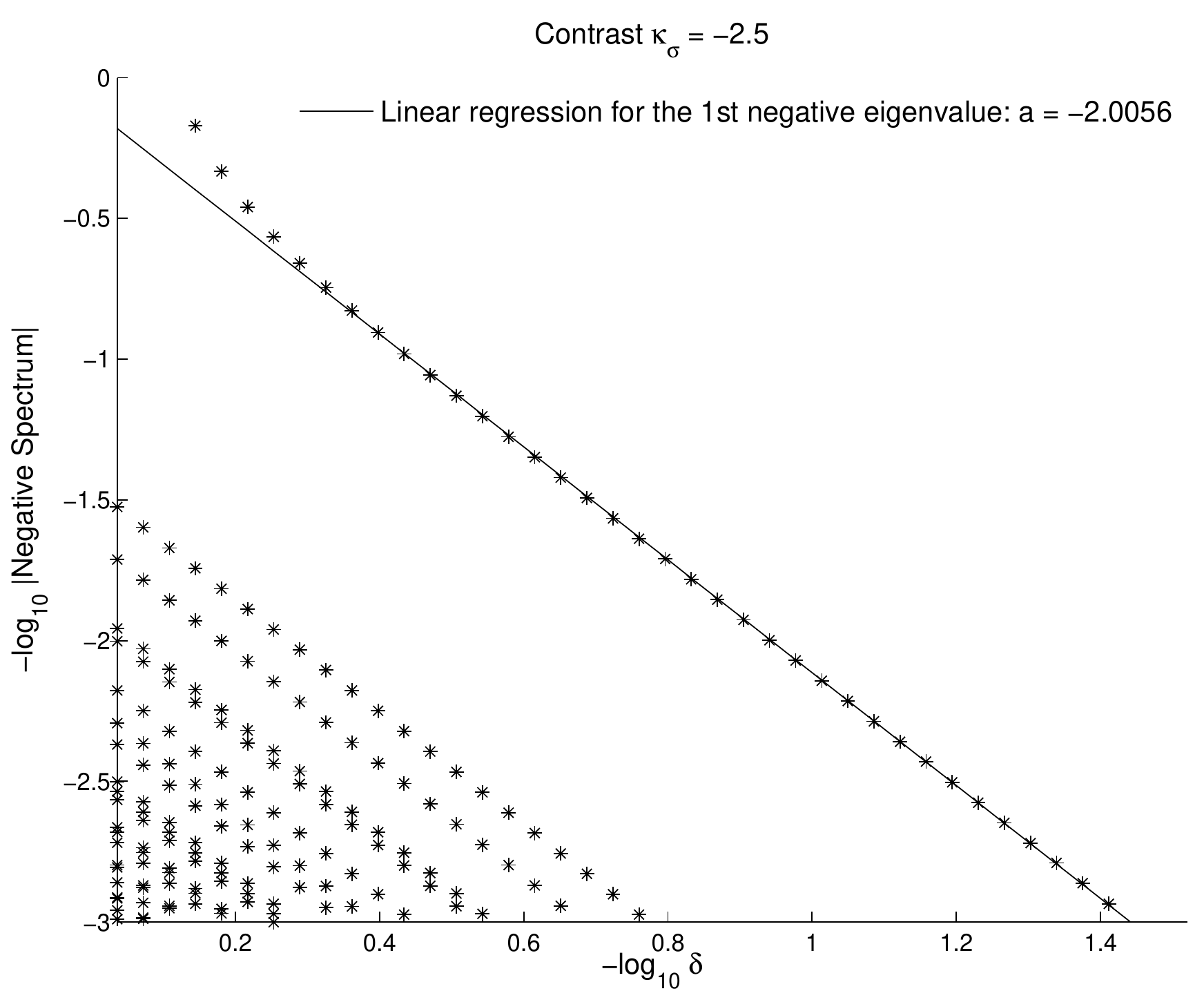}
\caption{The figure represents the approximation of the negative eigenvalues of smallest modulus of $\mA^{\delta}$ in logarithmic scale.}
\label{NegativeSpectrumRegression}
\end{figure}

\begin{figure}[!ht]
\centering
\begin{tabular}{cc}
\includegraphics[scale=0.06]{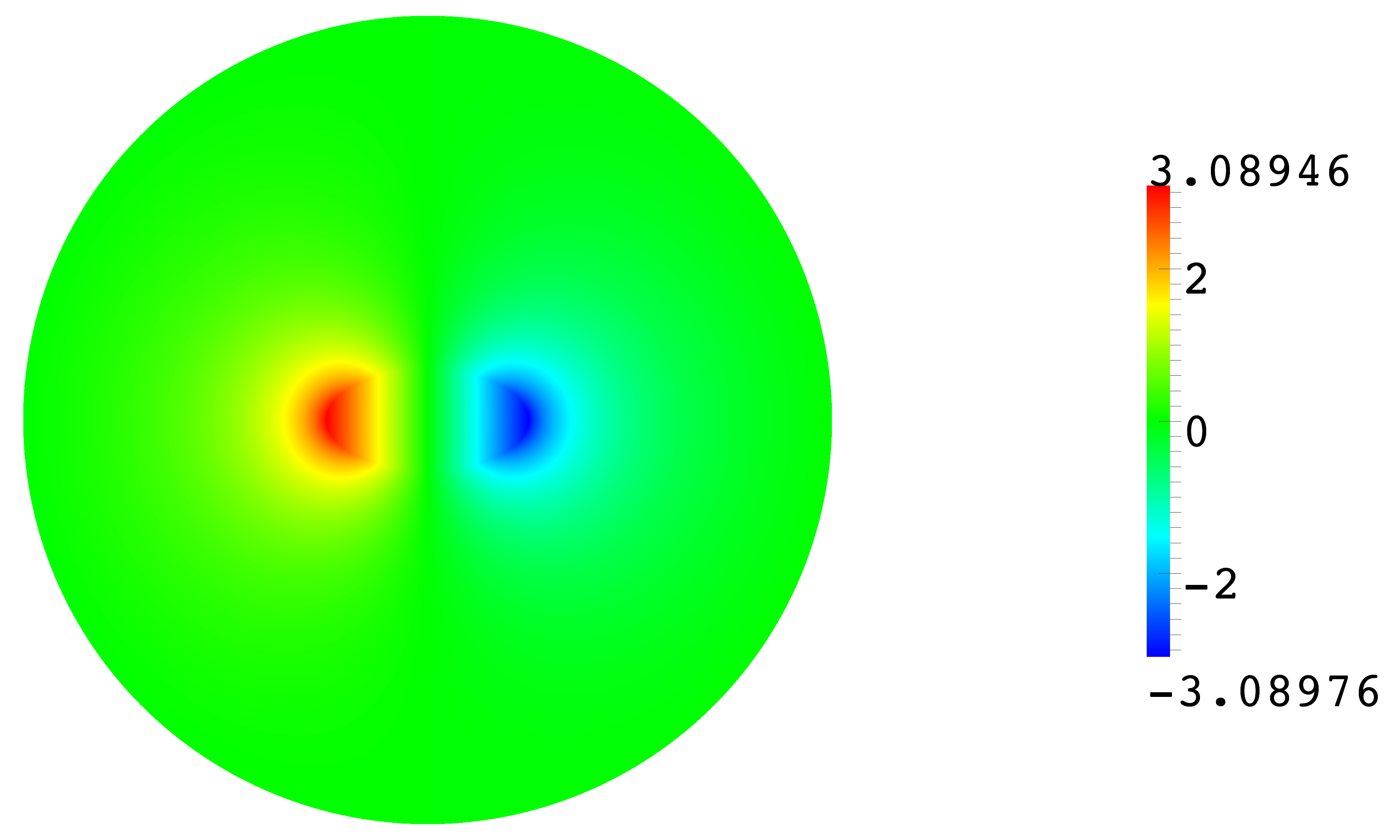}\quad&\quad
\includegraphics[scale=0.06]{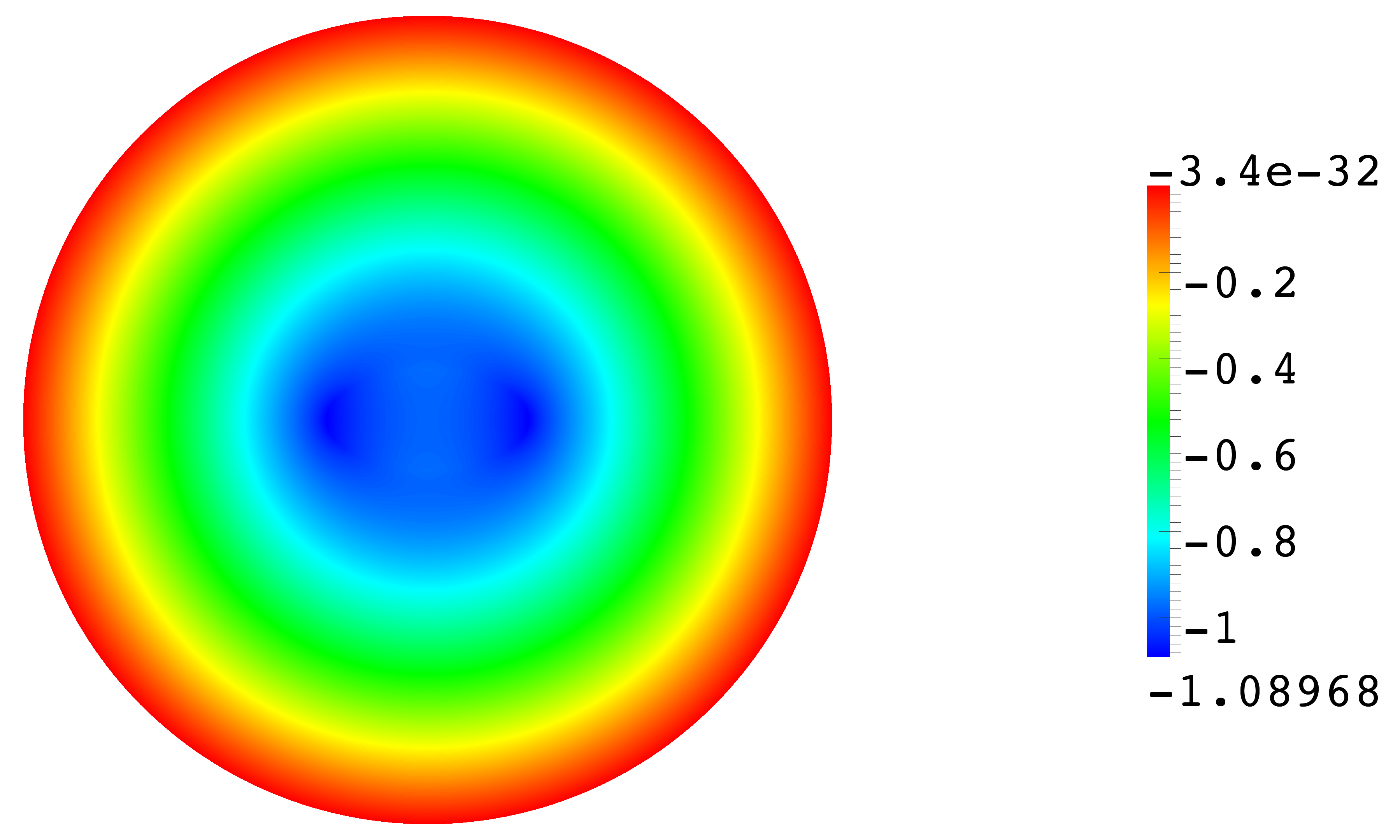}\\[6pt]
\includegraphics[scale=0.06]{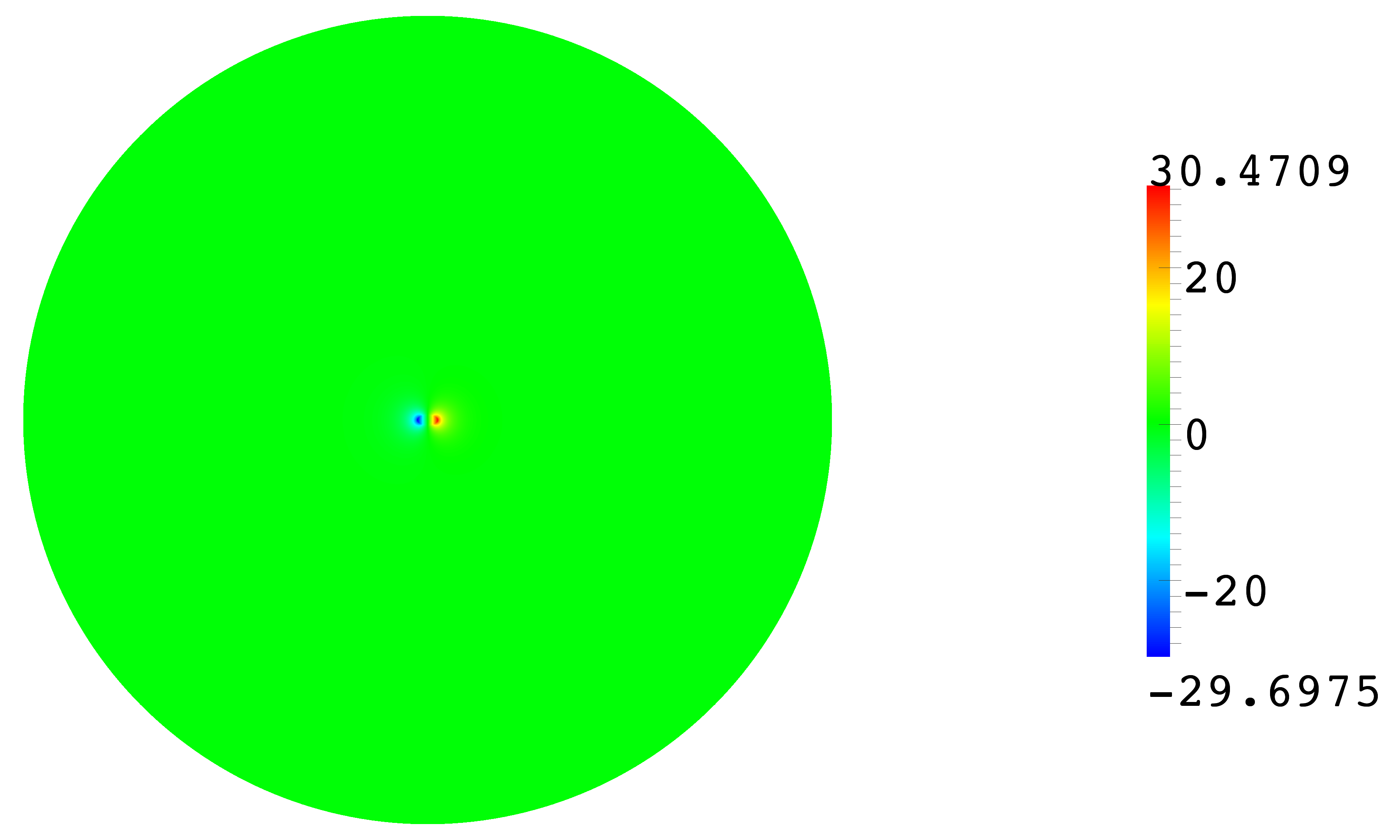}\quad&\quad
\includegraphics[scale=0.06]{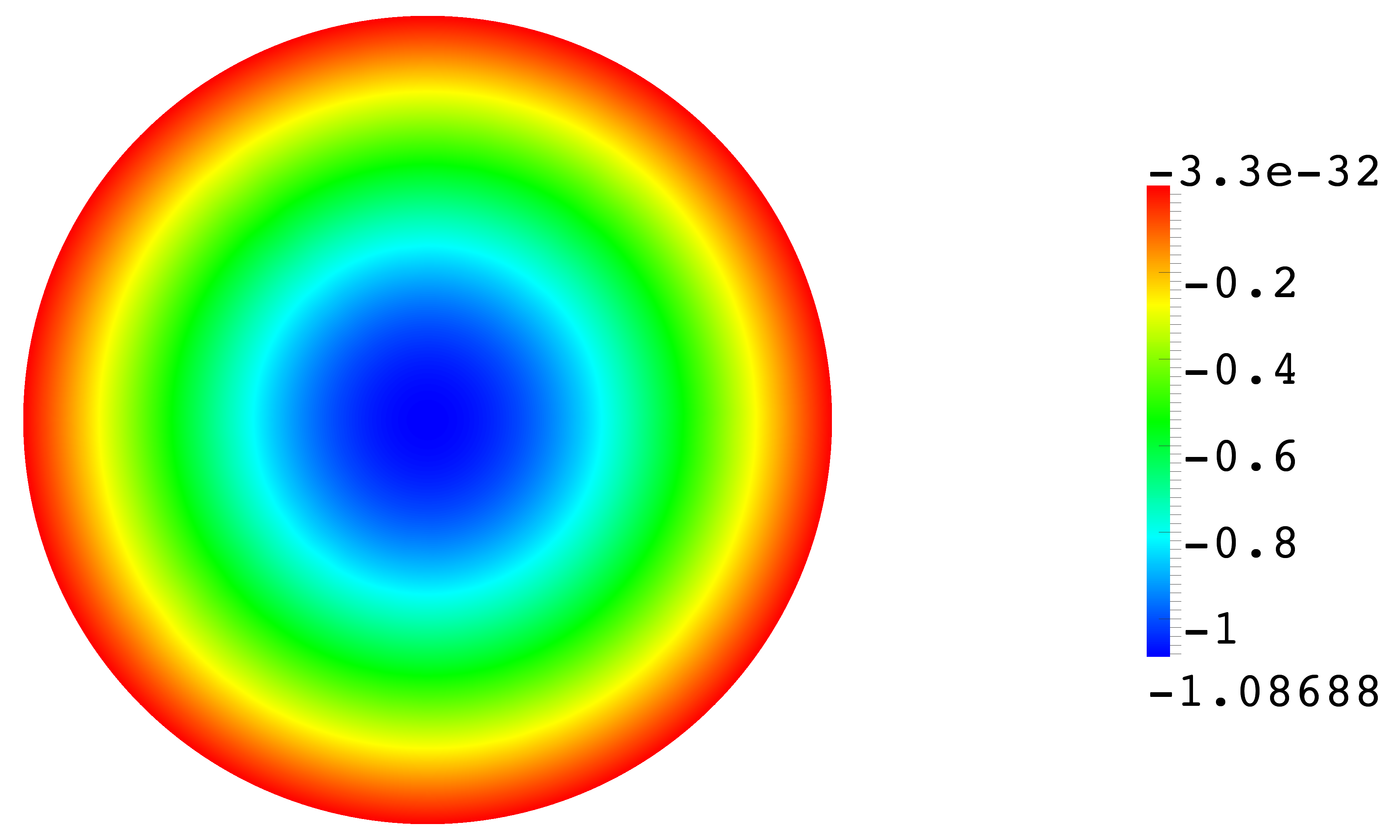}
\end{tabular}
\caption{On the left (resp. on the right), approximation of the eigenfunction associated with the negative (resp. positive) 
eigenvalue of smallest modulus. Above, $\delta=0.5$. Below, $\delta=0.05$. The contrast $\kappa_{\sigma}=\sigma_{-}/\sigma_{+}$ 
is chosen equal to $-2.5$. We observe clearly the localization effect for the eigenfunctions associated with the negative eigenvalues. 
\label{first eigenfunctions}}
\end{figure}

~
\newpage
~

\section*{Appendix} 
\noindent In this appendix, we briefly recall an elementary result of spectral theory that we used three times in this article in order to estimate the distance of a number to the spectrum of an operator.
We provide a proof for the sake of completeness.

\begin{lemma}\label{EstimEigenVal}
Let $\mrm{H}$, equipped with the inner product $(\cdot,\cdot)_{\mrm{H}}$ and the norm $\Vert \cdot\Vert_{\mrm{H}}$, be a Hilbert space. 
For any (\textit{a priori} unbounded) normal linear operator $\mrm{A}: D(\mrm{A})\subset \mrm{H}
\to\mrm{H}$ we have 
$$
\inf_{\lambda\in \mathfrak{S}(\mrm{A})}\vert \lambda - \mu\vert\leq \inf_{v\in D(\mA)\setminus\{0\}}
\frac{\Vert \mA v-\mu v \Vert_{\mrm{H}}}{\Vert v\Vert_{\mrm{H}}},\qquad\forall \mu\in \mathbb{C}.
$$
\end{lemma}
\begin{proof}
Since $\mA$ is normal then, according to the spectral theorem \cite[Thm 6.6.1]{BiSo87}, it admits a spectral 
decomposition $\mA = \int_{\mathfrak{S}(\mA)}\zeta d\mrm{E}(\zeta)$ where $\mrm{E}(\zeta)$ refers to a spectral 
measure on $\mrm{H}$. Let $d\mrm{E}_{v,v}$ refer to the measure associated with $\zeta\mapsto 
(\mrm{E}(\zeta)v,v)_{\mrm{H}}$. A spectral decomposition of  $\mA-\mu\mrm{Id}$ is given by
$\mA-\mu\mrm{Id} = \int_{\mathfrak{S}(\mA)}(\zeta - \mu) d\mrm{E}(\zeta)$. Moreover the formula 
$\Vert \mA v-\mu v\Vert_{\mrm{H}}^{2} = \int_{\mathfrak{S}(\mA)}\vert\zeta - \mu\vert^{2} d\mrm{E}(\zeta)$
holds for any $v\in D(\mA)$. As a consequence, we have  
$$
\Vert v\Vert_{\mrm{H}}^{2}\inf_{\lambda\in \mathfrak{S}(\mrm{A})}\vert \lambda - \mu\vert^{2} = 
\inf_{\lambda\in \mathfrak{S}(\mrm{A})}\vert \lambda - \mu\vert^{2}\int_{\mathfrak{S}(\mA)}d\mrm{E}_{v,v}(\zeta)\leq 
\int_{\mathfrak{S}(\mA)}\vert \zeta - \mu\vert^{2}d\mrm{E}_{v,v}(\zeta) = \Vert \mA v-\mu v\Vert_{\mrm{H}}^{2}.
$$
Since this holds for any $v\in D(\mA)$, we can divide by $\Vert v\Vert_{\mrm{H}}^{2}$ and take the $\inf$ in the 
right hand side of the estimate above, which yields the desired inequality.
\end{proof}

\section*{Acknowledgments} 
\noindent The research of the two first authors is supported by the ANR project METAMATH, grant 
ANR-11-MONU-016 of the French Agence Nationale de la Recherche. The work of the first 
author is also partially supported by the Academy of Finland (decision 140998) and by the FMJH through the grant ANR-10-CAMP-0151-02 in the ``Programme des Investissements d'Avenir''. The research of 
the third author is supported by the Russian Foundation for Basic Research, grant No. 12-01-00348.

\bibliography{Bibli}
\bibliographystyle{plain}

\end{document}